\documentclass[11pt,a4paper]{amsart}
\usepackage[utf8]{inputenc}
\usepackage[colorlinks=true,citecolor=blue]{hyperref}
\usepackage{color}
\usepackage{geometry}
\usepackage{stmaryrd}
\usepackage{mathabx}
\usepackage{graphicx}
\usepackage[foot]{amsaddr}

\newtheorem{theorem}{Theorem}[section]
\newtheorem{proposition}[theorem]{Proposition}
\newtheorem{corollary}[theorem]{Corollary}
\newtheorem{lemma}[theorem]{Lemma}
\newtheorem{remark}[theorem]{Remark}

\numberwithin{subsection}{section}
\numberwithin{subsubsection}{subsection}
\numberwithin{theorem}{section}
\numberwithin{equation}{section} 

\newcommand{\eps}{\varepsilon}

\newcommand\bbA{\mathbb{A}}
\newcommand\bbE{\mathbb{E}}
\newcommand\bbN{\mathbb{N}}
\newcommand\bbR{\mathbb{R}}

\newcommand\bbX{\mathbb{X}}
\newcommand\bbZ{\mathbb{Z}}

\newcommand{\mcF}{\mathcal{F}}

\newcommand{\bfW}{\mathbf{W}}
\newcommand{\bfX}{\mathbf{X}}

\DeclareMathOperator{\tr}{tr}
\newcommand{\dif}{{\mathrm d}}
\newcommand{\Ind}{{\bf 1}}
\newcommand{\E}{{\mathbf e}}
\newcommand{\I}{{\mathrm i}}
\DeclareMathOperator{\diag}{diag}

\title[Homogenisation for anisotropic kinetic random motions]{Homogenisation for anisotropic   \\
kinetic random motions}
\author[P. Perruchaud]{Pierre Perruchaud}
\address{Univ Rennes, CNRS, IRMAR - UMR 6625, F-35000 Rennes, France}
\email{pierre.perruchaud@univ-rennes1.fr}

\begin{document}

\begin{abstract}
We introduce a class of kinetic and anisotropic random motions $(x_t^{\sigma},v_t^{\sigma})_{t \geq 0}$ on the unit tangent bundle $T^1 \mathcal M$ of a general Riemannian manifold $(\mathcal M,g)$, where $\sigma$ is a positive parameter quantifying the amount of noise affecting the dynamics. As the latter goes to infinity, we then show that the time rescaled process $(x_{\sigma^2 t}^{\sigma})_{t \geq 0}$ converges in law to an explicit anisotropic Brownian motion on $\mathcal M$. Our approach is essentially based on the strong mixing properties of the underlying velocity process and on rough paths techniques, allowing us to reduce the general case to its Euclidean analogue. Using these methods, we are able to recover a range of classical results.
\end{abstract}

\maketitle

\section{Introduction}
\label{SectionIntroduction}

We consider a class of anisotropic and kinetic random motions on the unit tangent space of a general Riemannian manifold $(\mathcal M,g)$ of dimension $d \geq 2$. In the simplest case when the base manifold is the Euclidean space $\mathbb R^d$, the typical process we have in mind can be described as follows: let $\sigma>0$ be a positive parameter and let $(B_t)_{t \geq 0}$ be a Brownian motion in $\mathbb R^d$ with (non identity) covariance matrix $\Sigma=A^*A$. We construct an anisotropic diffusion process $(v_t)_{t \geq 0}=(v_t^{\sigma})_{t \geq 0}$ on the Euclidean sphere $\mathbb S^{d-1} \subset \mathbb R^d$ by solving the Stratonovich differential equation 
\begin{equation}\label{eq.defv}
\dif v_t = \displaystyle{\sigma \Pi_{v^{\perp}_t} \circ \dif B_t },
\end{equation}
where $\Pi_{v^{\perp}_t}$ denotes the projection on the orthogonal of $v_t$. We then integrate the velocity process $(v_t)_{t \geq 0}$ to obtain a process $(x_t)_{t \geq 0}=(x_t^{\sigma})_{t \geq 0}$ with values in $\mathbb R^d$
\begin{equation}\label{eq.defx}
x_t := x_0 +\int_0^t v_s \dif s.
\end{equation}
The process $(x_t,v_t)_{t \geq 0}$ is thus a diffusion process with values in the unit tangent space $T^1 \mathbb R^d=\mathbb R^d \times \mathbb S^{d-1}$. The first projection $(x_t)_{t \geq 0}$ is a $\mathcal C^1$ curve in $\mathbb R^d$, which inherits the anisotropy of the velocity process $(v_t)_{t \geq 0}$, and the positive parameter $\sigma$ allows one to slow or speed up the clock of the latter.
The next figure shows an approximation of a sample path of the resulting process.
\begin{figure}[ht]
\includegraphics[scale=0.3]{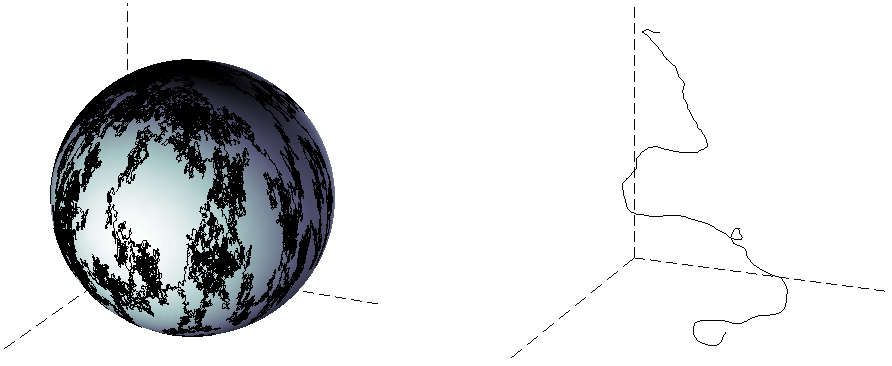}
\caption{A sample path of the velocity process $(v_t)_{0 \leq t \leq 10}$ on $\mathbb S^2$ (left) and the corresponding  $(x_t)_{0 \leq t \leq 10}$ in $\mathbb R^3$ (right) for the choice of covariance matrix $\Sigma=\diag(1,1.1,1.2)$.}
\end{figure}
\par
\medskip
On a general Riemannian manifold $(\mathcal M,g)$, an analogue process $(x_t,v_t)_{t \geq 0}$ with values in the unit tangent bundle $T^1 \mathcal M$ can be constructed starting from the above Euclidean process and using the classical stochastic development/parallel transport machinery. Namely, the process $(x_t,v_t)_{t \geq 0}$ in $T^1 \mathcal M$ is characterised by the fact that the image of $v_t \in T_{x_t}^1 \mathcal M$ in the fixed unit tangent space $T_{x_0}^1\mathcal M\simeq \mathbb S^{d-1}$ by the inverse stochastic parallel transport along $(x_s)_{0\leq s\leq t}$ solves equation \eqref{eq.defv} above. 

The isotropic analogue of the process, i.e. the process associated with $\Sigma=\mathrm{Id}$, was introduced in \cite{ABT} under the name kinetic Brownian motion, where its was shown that as the parameter $\sigma$ goes from zero to infinity, then the sample paths of the process $(x_{\sigma^2 t})_{t \geq 0}$ interpolates in a precise sense between geodesics and Brownian paths on the based manifold $\mathcal M$. For a fixed intensity parameter $\sigma$, the Poisson boundary of the process was also fully determined if the base manifold is rotationally invariant.
\par
\medskip

The motivation to introduce anisotropy in this context is twofold. From an applied point of view, the kinetic Brownian motion is a simple, yet very reasonable model for the dynamics of a mesoscopic spherical particle with bounded velocity in an isotropic heat bath. Compared to the standard Langevin dynamics where the velocities are Gaussian, the fact that the velocities are here of unit norm is perfectly consistent with special relativity theory. The homogenisation phenomena shown in \cite{ABT} illustrates the fact that the scaling limit of the process, i.e. the macroscopic behaviour of the particle is nevertheless diffusive, as awaited. Now, if the geometry of the mesoscopic particle under consideration is not spherical, or if the heat bath is anisotropic, the dynamics of the velocity process has to be anisotropic, see e.g. \cite{BHR, CP03, vKampen} and the references therein. In that context, the velocity evolution given by the stochastic differential equation \eqref{eq.defv} with $\Sigma \neq \mathrm{Id}$ is very natural. As we will see below and with this applied point of view, the main result of this article guarantees that the macroscopic behaviour of the particle is still diffusive, with an explicit anisotropy matrix.

From a more theoretical point of view,  the introduction of anisotropy is also unavoidable if one wants to generalise the results of \cite{ABT} to an infinite dimensional setting, say to an infinite dimensional Hilbert space. Indeed, doing so, one quickly faces the problem of defining spherical Brownian motion in this context. Looking at equation \eqref{eq.defv}, the orthogonal projection makes perfect sense in a Hilbert setting but we have to give meaning to the driving Brownian motion $B$. This can naturally be done using the notion of abstract Wiener space, see e.g. \cite{Gross, Gross2} or \cite[Chapter 8]{Stroock}. Roughly speaking, in that framework the driving process in \eqref{eq.defv} has to belong to the image of a radonifying injection, hence introducing a Hilbert-Schmidt covariance operator. In a finite dimensional setting, the action of this Hilbert-Schmidt operator amounts to replacing the standard Brownian motion $B$ by a Brownian motion with covariance $\Sigma \neq \mathrm{Id}$, i.e. to replace the isotropic noise driving kinetic Brownian motion by an anisotropic noise; this justifies our choice of dynamics for the velocity process.
\par
\medskip
Our goal in this paper is to exhibit the asymptotics of the time rescaled process $(x_{\sigma^2t}^{\sigma},v_{\sigma^2t}^{\sigma})_{t \geq 0}$ as the intensity parameter $\sigma$ goes to infinity. More precisely, we show that in both Euclidean and Riemannian contexts, its first projection converges to an anisotropic Brownian motion.
The presence of anisotropy drastically complexifies the approach and computations compared to the isotropic framework. Namely, in the isotropic Euclidean setting considered in Section 2.2 of \cite{ABT} and which is the core of the proof when associated with rough paths techniques, the homogenisation of kinetic Brownian motion was proved using It\^o calculus and standard martingale techniques. As it will be clear in Section 2 below, the Doob--Meyer decomposition of the velocity process given by equation \eqref{eq.defv} gets more involved here, its invariant measure is not likely to be easy to describe, and martingale techniques need explicit solutions of the Poisson equation which seems hopeless in this context. In fact, guessing a formula for the invariant measure of the $v^\sigma$ on the sphere before reading the statement of Proposition \ref{prop.mixing} does not seem obvious.
\par
\medskip
For this reason, we adopt a different approach and point of view here.  Our proof of  homogenisation for the time rescaled version of the process $(x_t^{\sigma},v_t^{\sigma})_{t \geq 0}$ is indeed essentially based on quantitative mixing properties of the velocity process. We show in particular that

\begin{proposition}[Lemma \ref{lem.mesinv} and Proposition \ref{prop.exponentialdec} below]\label{prop.mixing}
The process $v^\sigma_t$ solution of \eqref{eq.defv} is ergodic in $\mathbb S^{d-1}$ with an explicit invariant measure $\mu$ whose density with respect to the uniform measure $\dif\theta$ on the sphere is given by
\[
\frac{\dif\mu}{\dif\theta}(\theta)=\frac{\|A^{-1} \theta \|^{1-d}}{\int_{\,\mathbb S^{d-1}} \|A^{-1} \theta \|^{1-d} \dif\theta}.
\]
In particular, the invariant measure $\mu$ are well as the trajectories are invariant under all the coordinate reflections 
\begin{equation}\label{eq.symm}
(\theta^1, \cdots, \theta^i, \cdots, \theta^d) \mapsto (\theta^1, \cdots, -\theta^i, \cdots, \theta^d), \quad1 \leq i \leq d.
\end{equation}
Moreover, there exists a positive constant $\tau$ such that, if  $\mcF_{[a,b]}$ denotes the $\sigma$-algebra generated by the unit speed ($\sigma=1$) velocity process $v_t$, for $a\leq t<b$, then for any $0\leq s<t$ and any bounded measurable real-valued random variables $P$ and $F$ that are $\mcF_{[0,s]}$ and $\mcF_{[t,\infty]}-$measurable, respectively, we have
\begin{equation}\label{eq.mix}
\big|\bbE_\mu[PF] - \bbE_\mu[P]\,\bbE_\mu[F]\big| \lesssim |P|_\infty |G|_\infty\,\E^{-(t-s)/\tau}.
\end{equation}
\end{proposition}

The above strong mixing and symmetry properties of the velocity process are the key ingredients to establish the homogenisation of the anisotropic version of kinetic Brownian motion in the Euclidean setting. Indeed, we have the following result. 

\begin{theorem}[Proposition \ref{prop.convergencewk} and Theorem \ref{thm.convergenceRP} below]\label{thm.euclidean}
Let $(x_t^{\sigma},v_t^{\sigma})_{t \geq 0}$ with values in $T^1\mathbb R^d$ be the solution of equation \eqref{eq.defv} and \eqref{eq.defx}, starting from $(x_0,v_0)$ where $x_0$ is fixed and $v_0$ chosen at random according to $\mu$. Then as $\sigma$ goes to infinity, the time rescaled process $(x_{\sigma^2 t}^{\sigma})_{t\in[0,1]}$ converges in law to a Brownian motion in the Euclidean space $\mathbb R^d$, with covariance matrix $\diag(\gamma_1, \cdots, \gamma_d)$ where 
\[
\gamma_i:=2 \int_0^{+\infty} \bbE_\mu[v_0^i v_t^i] \dif t, \quad 1 \leq i \leq d.
\]
\end{theorem}

Our strategy of proof consists in establishing that the rough path lift of $(x_{\sigma^2 t}^{\sigma})_{t \geq 0}$ converges to the Stratonovich rough path lift of a Brownian motion with the above covariance. To do so, we use again the strong mixing properties of the velocity process, associated with a  Lamperti-type criterion to ensure the tightness of the lift in rough path topology ---  see Lemmas \ref{lem.kolmogorovwk} and \ref{lem.kolmogorovRP} below. We then identify the limit process by showing that it has to be a stationary process with independent Gaussian increments on the nilpotent group associated with the rough path structure, see Theorem \ref{thm.convergenceRP}. 

Using the fact that the notion of stochastic development amounts to solving a stochastic differential equation and that the It\^o map is continuous with respect to the rough paths topology, one can conclude that the previous Euclidean statement actually holds on a general Riemannian manifold. Anisotropic Brownian motion on $\mathcal M$ is defined as the stochastic development of an anisotropic Brownian motion in $T_{x_0}\mathcal M$.

\begin{theorem}\label{thm.manifold}
Let $(\mathcal M,g)$ be a complete and stochastically complete Riemannian manifold and let $(x_t^{\sigma},v_t^{\sigma})_{t \geq 0}$ be the process with values in $T^1 \mathcal M$ characterised by the fact that the image of $v_t \in T_{x_t}^1 \mathcal M$ in the fixed unit tangent space $T_{x_0}^1\mathcal M\simeq \mathbb S^{d-1}$ by the inverse stochastic parallel transport along $(x_s)_{0\leq s\leq t}$ solves equation \eqref{eq.defv} in $T_{x_0}^1 \mathcal M$. Then as $\sigma$ goes to infinity, the time rescaled process $(x_{\sigma^2 t}^{\sigma})_{t\in[0,1]}$ converges in law to an anisotropic Brownian motion on the base manifold $\mathcal M$.
\end{theorem}

As it will be clear from the proof of Theorem \ref{thm.euclidean}, the homogenisation phenomenon holds as soon as the mixing properties of the velocity process and the symmetry of the trajectories described in Proposition \ref{prop.mixing} hold. In other words, the conclusion of Theorem \ref{thm.manifold} is valid as soon as the process $(x_t^{\sigma},v_t^{\sigma})_{t \geq 0}$ we consider is the stochastic development of a velocity process satisfying the conclusions of Proposition \ref{prop.mixing}. In particular, our proof actually applies even if $(v_t)_{t \geq 0}=(v_t^{\sigma})_{t \geq 0}$ is an ergodic Markov process with jumps on $\mathbb S^{d-1}$ as soon as the conditions \eqref{eq.symm}  and \eqref{eq.mix} are fulfilled.

\begin{theorem}\label{thm.extension}
Let $(\mathcal M,g)$ be a complete and stochastically complete Riemannian manifold and let $(x_t^{\sigma},v_t^{\sigma})_{t \geq 0}$ be the process with values in $T\mathcal M$ characterised by the fact that the image of $v_t \in T_{x_t}\mathcal M$ in the fixed tangent space $T_{x_0}\mathcal M\simeq \bbR^d$ by the inverse stochastic parallel transport along $(x_s)_{0\leq s\leq t}$ satisfies the conditions \eqref{eq.symm}  and \eqref{eq.mix}. Then as $\sigma$ goes to infinity, the time rescaled process $(x_{\sigma^2 t}^{\sigma})_{t\in[0,1]}$ converges in law to an anisotropic Brownian motion on the base manifold $\mathcal M$.
\end{theorem}

See Theorem \ref{thm.final} for a precise statement. In this level of generality, in Section \ref{SubsectionExamples} we recover classical results, amongst which Pinsky's so-called random flight \cite{Pinsky} and time-dependent variations of it; the anisotropic Langevin diffusion, where $v$ is an anisotropic Ornstein-Uhlenbeck process; and linear interpolation of symmetric random walks as in \cite{Breuillard}. It is unclear whether or not the methods of X.M. Li \cite{XueMei1, XueMei2} or Herzog, Hottovy and Volpe \cite{HHV16} can get back such a result. In a somewhat independent direction, the interesting work \cite{ChevyrevMelbourne} of Chevyrev and coauthors studies this kind of convergence in deterministic systems.
\par
\medskip
The outline of the article is the following. In the next Section \ref{SectionMixingVelocity}, we study the velocity process solution of equation \eqref{eq.defv}. We characterise its invariant measure and establish the mixing properties which are the key ingredients in our approach of the homogenisation phenomenon. Section \ref{SectionProof} is then devoted to the proofs of our main Theorem \ref{thm.euclidean} and \ref{thm.manifold}. More precisely, in Section \ref{SubsectionTightness}, we show the tightness of the rough path lift of the process in the Euclidean setting. In Section \ref{SubsectionLimit}, we then identify the limit as a Brownian motion on the underlying two-step nilpotent Lie group. This completes the proof of Theorem \ref{thm.euclidean} in the Euclidean setting. Finally, in Section \ref{SubsectionManifold}, we use the continuity of the It\^o map to extend the proof of homogenisation to an arbitrary complete stochastically complete Riemannian manifold. The last section consists in developments, including Theorem \ref{thm.extension} and comments in Section \ref{SubsectionFinalTheorem}, and various examples in Section \ref{SubsectionExamples}.

\bigskip

\section{Mixing properties of the velocity process}
\label{SectionMixingVelocity}

Let $(B_t)_{t \geq 0}$ be a Euclidean Brownian motion in $\mathbb R^d$ with non degenerate covariance matrix $\Sigma$. Without loss of generality, up to an appropriate choice of coordinate system, we can assume that the matrix $\Sigma$ is diagonal, with square root $A$, namely
\[
\Sigma = \diag\left ( \alpha_1^2, \cdots,  \alpha_d^2 \right), \quad A = \diag \left( \alpha_1, \cdots, \alpha_d  \right).
\]
Let us recall that, by definition, the anisotropic velocity process $(v_t)=(v^1_t, \cdots, v^d_t)$ with values in $\mathbb S^{d-1} \subset \mathbb R^d$ and with intensity $\sigma>0$ is the solution of the Stratonovich stochastic differential equation
\[
\dif v_t = \displaystyle{\sigma \Pi_{v^{\perp}_t} \circ \dif B_t },
\]
where $\Pi_{v^{\perp}_t}$ denotes the projection on the orthogonal of $v_t$. Equivalently, there exist a standard Euclidean Brownian motion $(W_t)_{t \geq 0}$ such that $v_t$ satisfies the It\^o stochastic differential equation
\[ 
\dif v_t = \sigma \Pi_{v^{\perp}_t} A \dif W_t -\frac{\sigma^2}{2} \big(  \Sigma + \tr (\Sigma) \mathrm{Id}  - 2 \langle v_t, \Sigma v_t \rangle \mathrm{Id} \big) v_t \dif t,
\]
or even more explicitly in Euclidean coordinates, for $1 \leq i \leq d$
\begin{equation}\label{eq.coord}
\begin{array}{rll}
\displaystyle{\dif v_t^i } & =& \displaystyle{ - \frac{\sigma^2}{2} v^i_t\left[  \alpha_i^2 +\sum_{j=1}^d \alpha_j^2  - 
 2 \sum_{j=1}^d \alpha_j^2 |v^j_t|^2\right]\dif t}
\displaystyle{  + \sigma \left( \alpha_i \dif W^i_t - v^i_t \sum_{j=1}^d\alpha_j v^j_t \dif W^j_t\right)}.
\end{array}
\end{equation}

\medskip

In the following, $d$ and $\Sigma$ are fixed, and we write $f\lesssim g$ for some quantities $f$ and $g$ whenever $f\leq Cg$ for a constant $C>0$ independent of any other parameter.

\bigskip

\subsection{Invariant measure}
\label{SubsectionInvariantMeasure}

The object of this section is to establish that the velocity process $(v_t)_{t \geq 0}$ is ergodic in $\mathbb S^{d-1}$ and to write down its invariant measure explicitly. From equation \eqref{eq.coord}, it is not difficult to express the infinitesimal generator $L$ of the process and try to solve the equation $L^{*} \mu=0$. Nevertheless, since we are working on the sphere, integrations by parts and computations are quite unpleasant, and we prefer to introduce a natural Euclidean lift of the velocity process. Namely, if $\|\cdot\|$ denotes the standard Euclidean norm, consider the $\mathbb R^d$-valued process $(u_t)_{t \geq 0}$ starting from $u_0 \neq 0$ such that $v_0=u_0/\|u_0\|$, and solution of the stochastic differential equation system
\[ 
\dif u_t^i = \frac{\sigma^2}{2} \left(  - u^i_t \|u_t\|^2+\alpha^2_i u_t^i \right) \dif t + \sigma \alpha_i \|u_t\| \dif W^i_t, \quad 1 \leq i \leq d.
\]
Equivalently, it is the solution to the Stratonovich stochastic differential equation
\[
\dif u_t = -\frac{\sigma^2}2\|u_t\|^2u_t\,\dif t + \sigma\|u_t\|\circ\dif B_t.
\]
Then, a direct application of It\^o's formula shows that the projection $u_t/\|u_t\|$ on $\mathbb S^{d-1}$ satisfies equation \eqref{eq.coord}. To show that $u_t$ is ergodic and find an explicit expression for its invariant measure, let us now perform the simple linear change of variable $y_t :=A^{-1} u_t = \Sigma^{-1/2} u_t$. By It\^o's formula we get 
\[
\dif y_t^i = \frac{\sigma^2}{2} \left( -\| A y_t\|^2 y_t^i + \alpha_i^2 y_t^i \right) \dif t + \sigma \| A y_t\| \dif W_t.
\]
Setting  $V_A(y) := -\log \|A y\| + \frac{1}{2} \| y\|^2$, the infinitesimal generator $L_y$ of $y_t$ is given by
\[
L_y = \frac{\sigma^2}{2} \|A y\|^2 L_0, \quad \text{where} \quad 
L_0 :=\left(  -\nabla V_A \cdot \nabla + \Delta \right).
\]
The diffusion process with generator $L_0$ is naturally ergodic with  invariant measure proportional to $\E^{-V_A}$ so that $(y_t)_{t \geq 0}$ is also ergodic with invariant measure 
\[
\nu (\dif y) := C_{A} \,  \|A y\|^{-1} \E^{-\frac{1}{2} \|y\|^2}\,\dif y,
\]
where $C_{A}$ is a normalizing constant. In other words, the Euclidean lift $(u_t)_{t \geq 0}$ of $(v_t)_{t\geq 0}$ is ergodic in $\mathbb R^d$ and its invariant measure is proportional to $\|\cdot\|^{-1}$ times the centred Gaussian measure with covariance $\Sigma$. One can then compute the invariant measure of the velocity process as the image measure of the latter with respect to the projection on the sphere.

\begin{lemma}\label{lem.mesinv}
The velocity process $(v_t)_{t \geq 0}$ is ergodic in $\mathbb S^{d-1}$ and its invariant measure $\mu$ is absolutely continuous with respect to the uniform measure $\dif\theta$ on the sphere, with a density given by
\[
\frac{\dif\mu}{\dif\theta}(\theta)=\frac{\|A^{-1} \theta \|^{1-d}}{\int_{\,\mathbb S^{d-1}} \|A^{-1} \theta \|^{1-d} \dif\theta}.
\]
In particular, the invariant measure $\mu$ of the velocity process is invariant under all the coordinate reflections $(\theta_1, \cdots, \theta_i, \cdots, \theta_d) \mapsto (\theta_1, \cdots, -\theta_i, \cdots, \theta_d)$, for $1 \leq i \leq d$.
\end{lemma}

\begin{proof}
For any bounded measurable test function $f$ on $\mathbb S^{d-1}$, we have
\[
\begin{array}{ll}
\displaystyle{  \int_{\,\mathbb S^{d-1}}  f(v)\mu(\dif v)} & = \displaystyle{C_{A}  \int_{\mathbb R^d}  f\left(\frac{A y}{\| A y\|}\right)\frac{\E^{-\frac{1}{2} \|y\|^2} }{\|A y\|} \dif y}\\
\\
& = \displaystyle{C_{A}  \int_{\mathbb R^d}  f\left(\frac{u}{\| u\|}\right)\|u\|^{-1} \E^{-\frac{1}{2}  \|A^{-1} u\|^2} \frac{\dif u}{\det A}}\\
\\
 & = \displaystyle{C'_A \int_0^{+\infty} \int_{\,\mathbb S^{d-1}}  f\left(\theta\right)r^{-1}\E^{-\frac{1}{2} r^2 \|A^{-1} \theta \|^2} r^{d-1} \dif r \dif\theta}\\
 \\
  & = \displaystyle{\frac{\int_{\,\mathbb S^{d-1}} f(\theta) \|A^{-1} \theta \|^{1-d} \dif\theta}{\int_{\,\mathbb S^{d-1}} \|A^{-1} \theta \|^{1-d} \dif\theta}}.
\end{array}
\]
\end{proof}

The next figures illustrate the relation between the covariance matrix $\Sigma$, the sample paths of the velocity process $(v_t)$ and its invariant measure $\mu$.
The colour map on the sphere is chosen according to the value of the density of the invariant measure: small values of $\|A^{-1} \theta \|^{1-d}$ are represented in light grey whereas large values are represented in dark grey.

\begin{figure}[ht]\par
\includegraphics[width=350pt]{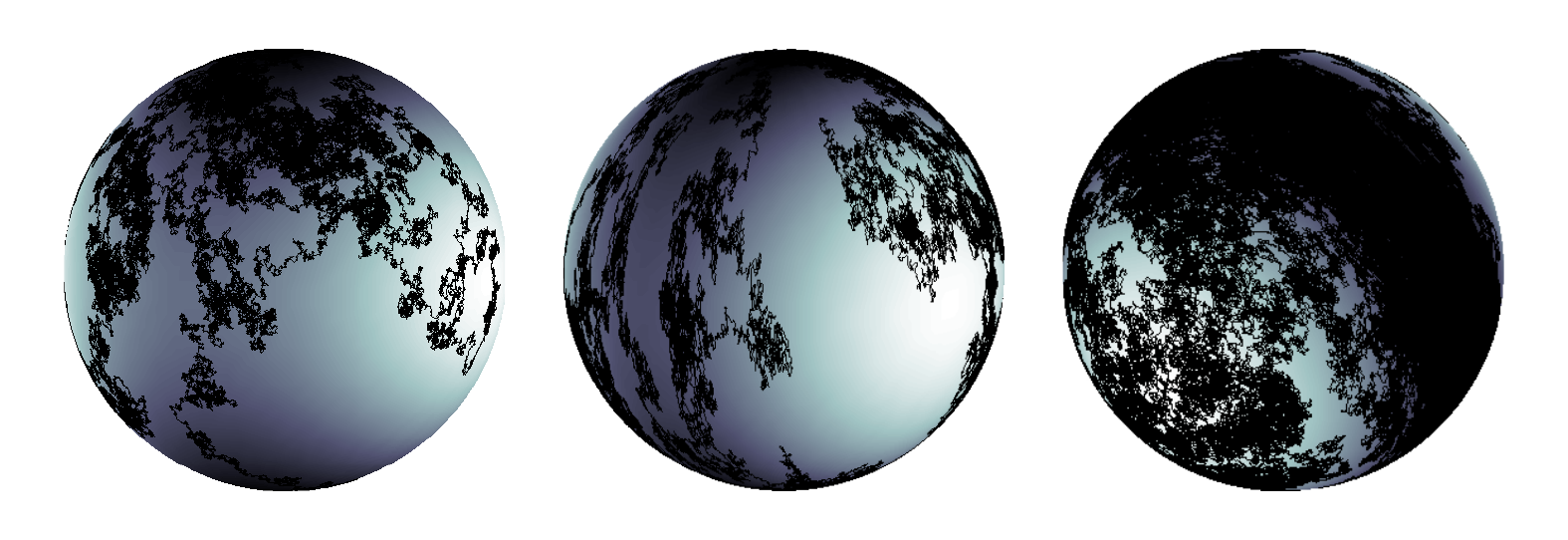}\par
\vspace{-0.3cm}
\caption{From left to right, sample paths of the velocity process and colour map of the invariant probability measure for $\Sigma=\diag(1,1.1,1.2)$, $\Sigma=\diag(1,4,9)$, and $\Sigma=\diag(1,100,100)$.}
\end{figure}

\begin{remark}
Let us emphasise here that the invariant measure $\mu$ of the velocity process actually differs from the projected Gaussian measure with covariance $\Sigma$, also known as angular Gaussian distribution, which, at first sight, could seem like a natural candidate for the velocity's equilibrium measure. Namely, if $f$ is a bounded measurable test function on the sphere, and if $X$ is a Gaussian variable in $\mathbb R^d$ with law $\mathcal N(0 ,\Sigma)$, we have indeed
\[
\begin{array}{ll}
\mathbb E \left[ f\left( \frac{X}{\|X\|}\right)\right] 
& = \displaystyle{\frac{\int_{\,\mathbb S^{d-1}} f(\theta) \|A^{-1} \theta \|^{-d} d\theta}{\int_{\,\mathbb S^{d-1}} \|A^{-1} \theta \|^{-d} d\theta}}.
\end{array}
\]
In other words, the invariant measure $\mu$ admits a density proportional to $\|A^{-1}\theta\|$ with respect to the standard projected Gaussian measure of covariance $\Sigma$.
\end{remark}

\begin{remark}
Going back to the modelisation point of view mentioned in the introduction, where $(v_t)_{t \geq 0}$ is thought as the velocity of a mesoscopic particle in an anisotropic heat bath, the invariant measure $\mu$ also differs from the standard choices for equilibrium measure in directional statistics, such as the Von Mises--Fisher distribution, Fisher--Bingham distribution or wrapped Brownian distributions, see Sections 9.3 and 9.4 of \cite{Mardia} and the references therein. We emphasise here the fact that the dynamics governed by equation \eqref{eq.defv} is fully intrinsic so that the measure $\mu$ is a simple and natural candidate to model anisotropic data; it also has natural interpretation in terms of projection of the invariant measure of the Euclidean lift $(u_t)_{t \geq 0}$.
\end{remark}

\bigskip

\subsection{Mixing properties}
\label{SubsectionMixingProperties}

Let us now establish the strong mixing properties of the velocity process that will be our main tool in the proof of the homogenisation result, Theorem \ref{thm.euclidean}. To avoid changes in the time scale, we fix $\sigma = 1$, from here to the end of the section. We also introduce a few additional notations. If $\lambda$ is a probability distribution on $\mathbb S^{d-1}$, let $\mathbb P_\lambda$ be a probability measure under which the velocity $(v_t)_{t\geq0}$ solves equation \eqref{eq.defv} with initial condition $v_0\sim\lambda$, and $\mathbb E_\lambda$ its associated expectation. We denote by $(P_t)_{t\geq0}$ the semigroup associated to $v$, acting on continuous functions $f:\mathbb S^{d-1}\to\bbR$, and by $(P_t^*)_{t\geq0}$ its dual, acting on probability measures on $\mathbb S^{d-1}$. In other words,
\[ P_tf(x) := \bbE_{\delta_x}[f(v_t)]
   \qquad\text{ and }\qquad
   P_t^*\lambda := \mathcal L(v_t|v_0\sim\lambda), \]
for any such $f$ and $\lambda$.

\medskip

To get to the second part of Proposition \ref{prop.mixing}, we use the well-known fact that since the velocity process $(v_t)_{t \geq0} $ is an elliptic diffusion in a compact Riemannian manifold, here the unit sphere, with invariant probability measure $\mu$, we have the estimate
\begin{equation}
\label{eq.exponentialconv}
\|P_t^*\lambda - \mu\|_\mathrm{TV} \lesssim \exp(-t/\tau)
\end{equation}
for any probability $\lambda$ on $\mathbb S^{d-1}$, for some positive constant $\tau$. Given an interval $[a,b)$ of $[0,\infty)$, define $\mcF_{[a,b)}$ as the $\sigma$-algebra generated by the unit speed velocity process $v_t$, for $a\leq t<b$. We write $A\in \mcF_{[a,b)}$ to say that a random variable is $\mcF_{[a,b)}$-measurable.

\begin{proposition}
\label{prop.exponentialdec}
For any $0\leq s<t$ and any bounded measurable real-valued random variables $P\in\mcF_{[0,s)}$ and $F\in\mcF_{[t,\infty)}$, we have
\[
\big|\bbE_\mu[PF] - \bbE_\mu[P]\,\bbE_\mu[F]\big| \lesssim |P|_\infty |G|_\infty\,\E^{-(t-s)/\tau}.
\]
\end{proposition}

\medskip

\begin{proof}
Since 
\[   
\big|\bbE_\mu[PF]-\bbE_\mu[P]\bbE_\mu[F]\big| \leq |P|_\infty\bbE_\mu\Big[\big|\bbE_\mu[F|\mathcal F_{[0,s]}]-\bbE_\mu[F]\big|\Big],
\]
by the Markov property, it suffices to prove that one has 
\begin{equation}
\label{eq.Gdecorrelation}
\big|\bbE_{P_u^*\lambda}[G] - \bbE_\mu[G]\big| \lesssim |G|_\infty \E^{-u/\tau},
\end{equation}
for any probability measure $\lambda$ on $\mathbb S^{d-1}$ and any real-valued measurable functional $G$. By a monotone class argument, it suffices to prove estimate \eqref{eq.Gdecorrelation} for elementary functionals of the form $G = g(v_{t_1},\dots, v_{t_k})$, for some bounded continuous real-valued function $g$ on $(\bbR^d)^k$ and times $t_1\leq\cdots\leq t_k$. But since the diffusion has the Feller property, the function $\overline{g}(v_0):=\bbE_{v_0}\big[g(v_{t_1},\dots, v_{t_k})\big]$ is continuous on the sphere, so we get \eqref{eq.Gdecorrelation} in that case by applying  \eqref{eq.exponentialconv} to $\overline{g}$.
\end{proof}

\medskip

The remainder of the section is devoted to the proof of the technical Lemma \ref{lem.readytouse}, that states an estimate about iterated integrals involving the covariances between the coordinates of the unit speed velocity process. Given a collection of positive times $s_1,\dots,s_n$, set $\Delta:=\max_{1\leq k< n}(s_k\wedge s_{k+1})$. We denote by $k_0\in\llbracket 1,n-1\rrbracket$ an index where this maximum is attained.

\medskip

\begin{proposition}
\label{prop.mixingcoords}
Under $\mathbb P=\mathbb P_\mu$, and for any indices $1\leq j_1,\dots, j_n\leq d$ and times $s_1,\cdots,s_n\geq0$,
\[   
\left| \bbE\big[v_{s_1}^{j_1}\cdots v_{s_1+\cdots+s_n}^{j_n}\big]\right| \lesssim \E^{-\Delta/{\tau}}.
\]
\end{proposition}

\medskip

\begin{proof}
For $1\leq i\leq n$, set $t_i := s_1+\cdots+s_i$, and define the bounded quantities
\[ 
V_- := v^{j_1}_{t_1}\cdots v^{j_{k_0-1}}_{t_{k_0-1}}, \qquad  V_0 := v^{j_{k_0}}_{t_{k_0}},\qquad  V_+ := v^{j_{k_0+1}}_{t_{k_0+1}}\cdots v^{j_n}_{t_n}.
\]
Note that $V_0$ is centred. Applying Proposition \ref{prop.exponentialdec} twice, this decomposition gives
\begin{align*}
     \left| \bbE\big[v_{s_1}^{j_1}\cdots v_{s_1+\cdots+s_n}^{j_n}\big] \right|
 & = \big|\bbE[V_-V_0V_+] - \bbE[V_-]\bbE[V_0]\bbE[V_+]\big|   \\
 & \leq \big|\bbE[V_-V_0V_+]
      - \bbE[V_-]\bbE[V_0V_+]\big|
      + |V_-|_\infty\left|\bbE[V_0V_+]
      - \bbE[V_0]\bbE[V_+]\right| \\
 & \lesssim |V_-|_\infty|V_0V_+|_\infty
            \E^{-s_{k_0}/\tau}
          + |V_-|_\infty |V_0|_\infty|V_+|_\infty \E^{-s_{k_0+1}/\tau} \\
 & \lesssim \E^{-\Delta/\tau}.\qedhere
\end{align*}
\end{proof}

\bigskip

\begin{lemma}
\label{lem.readytouse}
Suppose $\mathbb P=\mathbb P_\mu$. Given a positive integer $n$, we have
\[   
\int_{0\leq t_1\leq \cdots \leq t_{2n}\leq T} \left|\bbE_\mu\big[v_{t_1}^{i_1}\cdots v_{t_{2n}}^{i_{2n}}\big]\right| \,\dif t_1\dots \dif t_{2n} \lesssim_n T^n 
\]
for any indices $1\leq i_1,\cdots, i_{2n}\leq d$.
\end{lemma}

\medskip

\begin{proof}
The idea is to apply Proposition \ref{prop.mixingcoords} with the largest $\Delta$ possible for each tuple $(t_1,\cdots,t_{2n})$. Write first
\[
\int_{0\leq t_1\leq \cdots \leq t_{2n-1}} \left|\bbE\big[v_{t_1}^{i_1}\cdots v_{t_{2n}}^{i_{2n}}\big]\right|\,\dif t_1\dots \dif t_{2n} \leq \int_{[0,T]^{2n}} \left|\bbE\big[v^{i_1}_{s_1}\cdots v^{i_{2n}}_{s_1+\cdots+s_{2n}}\big]\right|\, \dif s_1\dots \dif s_{2n}.
\]
Fix now the tuple $(s_1,\cdots,s_{2n})$, and set
\[ 
\Delta(s) := \max_{1\leq k< 2n}\left(s_k\wedge s_{k+1}\right),
\]
so the integrand in the right hand side above is bounded above by a constant multiple of $\E^{-\Delta(s)/\tau}$, from Proposition \ref{prop.mixingcoords}.

\smallskip

The rest is combinatorics. We first sort the indices $k$ of the gaps $s_k$ according to the value of $s_k$ with respect to $\Delta = \Delta(s)$. Set $a:=\min\{k\in\llbracket 1,2n\rrbracket\,: s_k=\Delta\}$. Then, note that there are at most $n$ gaps $s_k$ of size larger than $\Delta$: otherwise, two of them would be consecutive, and $\Delta$ would not be optimal. This is the same as saying that there are at least $n$ small gaps $s_k\leq\Delta$, including $s_a$. Define $1\leq b_1<\cdots<b_{n-1}\leq 2n$ as the first $(n-1)$ indices different from $a$ corresponding to gaps of size at most $\Delta$. In other words, if $s_k\leq\Delta$, then either $k=b_i$ for some $1\leq i< n$, $k=a$, or $k>a,b_{n-1}$. Finally, denote by $1\leq c_1<\cdots<c_n\leq 2n$ the other indices, so that we have a partition of $\{1,\cdots,2n\}$ in three sets $A(s) := \{a\}$, $B(s) := \{b_1,\cdots,b_{n-1}\}$ and $C(s) := \{c_1,\cdots,c_n\}$ of fixed sizes. Now, given a fixed partition $(\alpha,\beta,\gamma)$ of $\llbracket 1,2n\rrbracket$ with $\alpha=\{\alpha_0\}$ of size $1$, and the set $\beta=\{\beta_1,\dots,\beta_{n-1}\}$ of size $n-1$, we have
\begin{equation*}
\begin{split}
\left|\bbE\big[v^{i_1}_{s_1}\cdots v^{i_{2n}}_{s_1+\cdots+s_{2n}}\big]\right| \Ind_{(A(s),B(s),C(s))=(\alpha,\beta,\gamma)} &\lesssim \E^{-\Delta(s)/\tau}\Ind_{s_{\beta_1},\dots,s_{\beta_{n-1}}\leq\Delta(s)}   \\
&\lesssim \E^{-s_{\alpha_0}/\tau}\Ind_{s_{\beta_1},\dots, s_{\beta_{n-1}}\leq s_{\alpha_0}},
\end{split}
\end{equation*}
from which we get
\begin{align*}
\int_{[0,T]^{2n}} &\left|\bbE\big[v^{i_1}_{s_1}\cdots v^{i_{2n}}_{s_1+\cdots+s_{2n}}\big]\right| \Ind_{(A(s),B(s),C(s))=(\alpha,\beta,\gamma)}\,\dif s_1\dots \dif s_{2n} \\
&\lesssim T^n\int_0^T\E^{-s/\tau}s^{n-1}\dif s \lesssim_n\, T^n
\end{align*}
and the result of the lemma, by summing over the set of all partitions $(\alpha,\beta,\gamma)$ of $\llbracket 1,2n\rrbracket$ with the above size.
\end{proof}

\bigskip

\section{Proof of the main result}
\label{SectionProof}

Let us now describe how the mixing properties of the velocity process derived in Section \ref{SubsectionMixingProperties} imply the homogenisation for the time rescaled position process $(x_{\sigma^2 t}^{\sigma})_{t \geq 0}$, as $\sigma$ goes to infinity, in both Euclidean and Riemannian framework. As mentioned in the introduction, we will actually work with a rough path lift of the kinetic process. We refer the reader to \cite{FrizHairer, BailleulFlows} for gentle introductions to rough paths theory, and given $\gamma\in (0,1)$, we denote by $\textsf{RP}(\gamma)=\textsf{RP}^\gamma([0,1],\bbR^d)$ the set of weak geometric $\gamma$-H\"older rough paths.

\medskip

\textbf{Notations.} We are interested in the stationary case $\mathbb P:=\mathbb P_\mu$, where $\mu$ is the invariant measure of the velocity, as described in Lemma \ref{lem.mesinv}. Define $X^\sigma:t\mapsto x^\sigma_{\sigma^2t}$, so that we are interested in the limiting behaviour of $(X^\sigma_t)_{t\geq0}$. To make good use of the mixing properties of $v$ such as Proposition \ref{prop.exponentialdec} without having to change the time scale, from now on \emph{$(v_t)_{t\geq0}$ will always stand for $(v^\sigma_t)_{t\geq0}$ with $\sigma=1$.} With this convention, we can express the increments of $X^\sigma$ as
\[ X_t - X_s = \frac1{\sigma^2}\int_{\sigma^4s}^{\sigma^4t}v_u\dif u\text. \]
The process $X_\sigma$ being $\mathcal C^1$, it admits a canonical rough path lift $\bfX^\sigma=(X^\sigma,\bbX^\sigma)$, where $\bbX^\sigma$ is defined by
\[ \bbX^\sigma_{ts} := \int_s^t(X^\sigma_u-X^\sigma_s)\otimes\dif X^\sigma_u
                     = \frac1{\sigma^4}\int_{\sigma^4s}^{\sigma^4t}
                                       \int_{\sigma^4s}^uv_z\otimes v_u\dif z\,\dif u\text.\]

Our proof relies on the algebraic properties of rough paths. Namely, that in the 2-step nilpotent group $G\subset\bbR\oplus\bbR^d\oplus(\bbR^d)^{\otimes2}$, the process $\mathbf x^\sigma:t\mapsto (1,X^\sigma_t,\bbX^\sigma_{t0})$ has increments
\[ (\mathbf x^\sigma_s)^{-1}\mathbf x^\sigma_t = (1,X^\sigma_t-X^\sigma_s,\bbX^\sigma_{ts}), \]
which, using the above expressions, are measurable with respect to $\sigma\big((v_u)_{\sigma^4s\leq u <\sigma^4t}\big)$.

Recall that we write $f\lesssim g$ for some quantities $f$ and $g$ when there exists a positive constant $C>0$ depending on $\Sigma$ alone such that $f\leq Cg$. If $C$ is allowed to depend on a parameter, say $p$, we write $f\lesssim_p g$.

\bigskip

\subsection{Tightness in rough paths space}
\label{SubsectionTightness}

We first establish that the family of processes $(X^\sigma_t)$ and their rough paths lifts are tight for the corresponding topology. To do so, we use a standard Lamperti criterion, namely we have the following lemma.
\begin{lemma}
\label{lem.kolmogorovwk}
For every $a\geq1$,
\[ 
\sup_{\sigma>0}\ \mathbb{E}\big[|X^\sigma_t - X^\sigma_s|^a\big] \lesssim_a |t-s|^{a/2}. 
\]
\end{lemma}

\medskip

\begin{proof}
Given any positive time $T$ and any positive integer $n$, we show that one has
\begin{equation}\label{eq.kolmogorovsimple}
\mathbb{E}\left[{\left|\int_0^Tv_t\,\dif t\right|^{2n}}\right] \leq C_nT^n 
\end{equation}
for some positive constant $C_n$ depending only on $n$. The inequality of the lemma follows as a consequence since for any positive integer $n$ such that $2n\geq a$, we have
\[   
\mathbb{E}\big[{|X^\sigma_t - X^\sigma_s|^a}\big]   = \mathbb{E}\big[{|X^\sigma_{t-s}|^a}\big]
\leq \frac1{\sigma^{2a}}\,\mathbb{E}\left[{\left|\int_0^{\sigma^4(t-s)}v_u\,\dif u\right|^{2n}}\right]^{a/2n}
\leq C_n^{a/2n}(t-s)^{a/2} \text. 
\]
Given $T>0$ and $n\in\mathbb{N}^*$, we have 
\begin{align*}
     \mathbb{E}\left[{\left|\int_0^Tv_t\,\dif t\right|^{2n}}\right]
= &\ \mathbb{E}\left[{\left(\sum_{1\leq i\leq d}\left(\int_0^Tv^i_t\,\dif t\right)^2\right)^n}\right] \\
= &  \sum_{1\leq i_1,\cdots,i_n\leq d}\int_{[0,T]^{2n}}
     \mathbb{E}\Big[{v^{i_1}_{t_1}v^{i_1}_{t_2}\cdots v^{i_n}_{t_{2n-1}}v^{i_n}_{t_{2n}}}\Big]
     \,\dif t_1\cdots\,\dif t_{2n},
\end{align*}
with the following estimate for each individual term on the right hand side. Fix $1\leq j_k\leq d$, for $1\leq k\leq 2n$. For any permutation $\phi\in\mathfrak S_{2n}$, we have from Lemma \ref{lem.readytouse}
\begin{align*}
       \int_{[0,T]^{2n}}
       \mathbb{E}\Big[{v^{j_1}_{t_1}\cdots v^{j_{2n}}_{t_{2n}}}\Big]
       \Ind_{t_{\phi(1)}<\cdots<t_{\phi(2n)}}
       \,\dif t
   = & \int_{0\leq t_1\leq \cdots\leq t_{2n}\leq T}
       \mathbb{E}\Big[{v^{j_{\phi(1)}}_{t_1}\cdots v^{j_{\phi(2n)}}_{t_{2n}}}\Big]
       \,\dif t_1\cdots\,\dif t_{2n}   \\
\lesssim_n &\ T^n,
\end{align*}
from which the result of the Lemma follows by summation over $\phi$ and $j$.
\end{proof}

\medskip

We use the Hilbert-Schmidt norm $|\cdot|$ on $\bbR^d\otimes\bbR^d \simeq L(\mathbb{R}^d)\simeq \mathbb{R}^{d^2}$; it coincides with the Euclidean norm on $\mathbb{R}^{d^2}$.

\medskip

\begin{lemma}
\label{lem.kolmogorovRP}
For every $a>0$,
\[ 
\sup_{\sigma>0} \mathbb{E} \big[{|\mathbb X^\sigma_{ts}|^a}\big] \lesssim_a |t-s|^a. 
\]
\end{lemma}

\smallskip

\begin{proof}
As above, the inequality of the statement follows from an inequality of the form
\[   
\bbE\left[{\left|\int_{0\leq s\leq t\leq T}v_s\otimes v_t\, \dif s\dif t\right|^{2n}}\right] \leq C_nT^{2n},
\]
for some positive constant $C_n$ depending only on $n$. Fix $T>0$ and $n\in\bbN^*$, and set for $\ell\in\llbracket 1,d\rrbracket^{4n}$
\[ 
I_\ell := \int_{0\leq s_1\leq t_1\leq T}\dots\int_{0\leq s_{2n}\leq t_{2n}\leq T} \bbE\Big[{v^{\ell_1}_{t_1}v^{\ell_2}_{s_1}\cdots v^{\ell_{4n-1}}_{t_{2n}}v^{\ell_{4n}}_{s_{2n}}}\Big] \, \dif s_1\dif t_1\cdots \dif s_{2n}\dif t_{2n},
\]
so we have
\begin{align*}
   \bbE\left[{\left|\int_0^T\int_0^tv_s\otimes v_t\,\dif s\dif t\right|^{2n}}\right] &= \bbE\left[{\left(\sum_{1\leq i,j\leq d} \left(\int_0^T\int_0^tv^i_sv^j_t\dif s\dif t\right)^2 \right)^n}\right]   \\
&= \sum_{i,j\in\llbracket 1,d\rrbracket^n} I_{i*j}
\end{align*}
with $i*j=(i_1,j_1,i_1,j_1,\cdots,i_k,j_k,i_k,j_k)$. As in Lemma \ref{lem.kolmogorovwk}, estimating each $I_{i*j}$ using Lemma \ref{lem.readytouse} does the job.
\end{proof}

\medskip

One can then apply the Kolmogorov-Lamperti tightness criterion for rough paths stated in Corollary A.12 of \cite{FV10} to get the following result from Lemma \ref{lem.kolmogorovwk} and Lemma \ref{lem.kolmogorovRP}.

\medskip

\begin{corollary}
\label{cor.tightness}
Pick $1/3<\gamma<1/2$. The family $\left\{\mathcal L(\bfX^\sigma)\right\}_{\sigma>0}$ of distributions on $\mathsf{RP}(\gamma)$ is tight.
\end{corollary}

\bigskip

\subsection{Brownian limit}
\label{SubsectionLimit}

The family of processes $(X^\sigma_t)$ and their lifts being tight for the rough paths topology, in order to establish its convergence, we are left to identify the possible limit process. Our strategy here is to prove that the latter is necessarily a stationary process with independent Gaussian increments on the underlying nilpotent group, and therefore is a Brownian motion. Let us set 
\[ 
\gamma_i := 2\int_0^\infty \mathbb{E}[v_0^iv_t^i] \,\dif t\text. 
\]

\medskip

\begin{proposition}
\label{prop.convergencewk}
For every $\gamma<1/2$, the processes $X^\sigma$ converge in distribution in $\mathcal C^\gamma([0;1],\bbR^d)$ to the Brownian motion on $\bbR^d$ with covariance matrix $\diag(\gamma_1,\cdots,\gamma_d)$, as $\sigma$ goes to $\infty$.
\end{proposition}

\medskip

\begin{proof}

\emph{Stationarity and independence.} 
We first show that any $\bbR^d$-valued process $X$ whose law $\widehat{\mathbb P}$ is a limit point of $(\mathcal L(X^\sigma))_{\sigma>0}$ in $\mathcal C^\gamma([0;1],\bbR^d)$ as $\sigma$ tends to $\infty$ has stationary independent increments.

\smallskip

Indeed, since $v_0$ has distribution the invariant measure of the diffusion $v$, the increments of $X^\sigma$ are stationary for every $\sigma$, so the increments of $X$ are stationary as well. Fix now $0\leq s_1<t_1\leq\cdots\leq s_n<t_n\leq 1$, and bounded continuous functions $F_i:\bbR^d\to\bbR$, for $1\leq i\leq n$. Fix $\eps>0$ small enough. From a repetitive use of Proposition \ref{prop.exponentialdec}, as used in Proposition \ref{prop.mixingcoords}, we have
\[ 
\left| \mathbb{E}\left[{\prod_{1\leq i\leq n}F_i(X^\sigma_{t_i-\eps}-X^\sigma_{s_i})}\right] - \prod_{1\leq i\leq n}\mathbb{E}\left[F_i(X^\sigma_{t_i-\eps}-X^\sigma_{s_i})\right]\right| \lesssim_n |F_1|_{L^\infty}\cdots|F_n|_{L^\infty}\,\E^{-\sigma^4\eps/\tau} 
\]
for some positive constant $\tau$, and we see that
\[ 
\widehat{\mathbb E}\left[\prod_{1\leq i\leq n}F_i(X_{t_i-\eps}-X_{s_i})\right] = \prod_{1\leq i\leq n}\widehat{\mathbb E}\big[F_i(X_{t_i-\eps}-X_{s_i})\big], 
\]
sending $\sigma$ to $\infty$ along a proper subsequence. Using the boundedness and continuity of the functions $F_i$ and the continuity of the process $X$, we can send $\eps$ to $0$ and see that $X$ has independent increments. So $X$ is a Brownian motion; it has null mean since every $X^\sigma_1$ has null mean, and its covariance is given by the limit of the covariances of the $X^\sigma_1$.

\medskip

\emph{Covariance formula.} First, it follows from the identity 
\[ 
\mathcal L(v^1,\cdots, v^i,\cdots,v^n) = \mathcal L(v^1,\cdots,-v^i,\cdots,v^n)
\]
that different components of $X_1$ have null covariance since this is the case for different components of $X^\sigma_1$. Now, for $1\leq i\leq d$, we have
\begin{align*}
      \mathbb{E}\big[\big((X^\sigma_1)^i\big)^2\big]
   &= \frac1{\sigma^4}\int_0^{\sigma^4}\int_0^{\sigma^4}\mathbb{E}[v^i_sv^i_t]\,\dif s\,\dif t = \frac2{\sigma^4}\int_0^{\sigma^4}\int_t^{\sigma^4}
      \mathbb{E}[v^i_sv^i_t]\,\dif s\,\dif t   \\
   &= \frac2{\sigma^4}\int_0^\infty\int_0^\infty \Ind_{t+u\leq\sigma^4}
      \mathbb{E}[v^i_{t+u}v^i_t]\dif u\,\dif t = 2\int_0^\infty\left(1-\frac u{\sigma^4}\right)_+ \mathbb{E}[v^i_u v^i_0] \dif u
\end{align*}
with $(\cdot)_+$ the positive part. According to Proposition \ref{prop.exponentialdec}, the integrand is smaller than a constant multiple of $\exp(-u/\tau)$, uniformly on $\sigma$. It is integrable, so we see from Lebesgue dominated convergence theorem that the above variance tends to $\gamma_i$.
\end{proof}

\medskip

\begin{theorem}
\label{thm.convergenceRP}
Pick $1/3<\gamma<1/2$. The processes $\bfX^\sigma$ converge in law in $\mathsf{RP}(\gamma)$ to the Brownian rough path on $\bbR^d$ with covariance matrix $\diag(\gamma_1,\cdots,\gamma_d)$, as $\sigma$ goes to $\infty$.
\end{theorem}

\medskip

\begin{proof} 

\emph{G-valued Lévy process.}
As above, we first notice that any limit measure of the laws of $(\bfX^\sigma)_{\sigma>0}$ turns the canonical process on $\textsf{RP}(\gamma)$ into a random process with stationary independent increments, in the free nilpotent Lie group of step $2$, as a consequence of the corresponding property for $\bfX^\sigma$. The canonical process on the free nilpotent Lie group of step $2$ is thus a continuous Lévy process under any limit law, so, according to Hunt's theorem, we can identify the former from its generator. More specifically, such a process $Y$ is characterised by the action
\[ f \mapsto \lim_{t\to0}\frac1t\bbE_e[f(Y_t)-f(e)]\in\bbR \]
of its generator on smooth functions $f:G\to\bbR$ with compact support, where $e$ is the unit of $G$; see \cite[Theorems 5.3.3]{Applebaum} or \cite[Theorem 1.1]{Liao}.

\medskip

\emph{Generator.}
Let $\widehat{\mathbb P}$ be any limit point of the laws of $\bfX^\sigma$ on $\textsf{RP}(\gamma)$, and denote by $\mathbf X = (X,\mathbb X)$ its canonical variable. We know from Proposition \ref{prop.convergencewk} that $X$ is a Brownian motion $W$; denote by $\bfW=(W,\mathbb W)$ its canonical Stratonovich rough path lift, also defined on the space $(\textsf{RP}(\gamma),\widehat{\mathbb P})$. Since the velocity process $v = (v^1,\dots,v^d)$ and $(v^1,\dots, v^{i-1},-v^i,v^{i+1},\dots,v^d)$ have the same law for every $1\leq i\leq d$, for $v_0$ distributed according to the invariant measure $\mu$, the antisymmetric part $\mathbb{A}^{\!\bfX}_{ts} := \frac{1}{2}(\bbX_{ts}-{}^t\bbX_{ts})$ is centred for any $0\leq s\leq t\leq 1$. We also know from the uniform estimates proved in Lemmas \ref{lem.kolmogorovwk} and \ref{lem.kolmogorovRP} that
\begin{equation}
\label{eq.rappelmoments}
\widehat{\mathbb E}\big[|X_t|^2\big]\lesssim t, \quad \widehat{\mathbb E}\big[|\mathbb{A}^{\!\bfX}_{t0}|^2\big]\lesssim\widehat{\mathbb E}\big[|\mathbb{X}_{t0}|^2\big]\lesssim t^2, 
\end{equation}
uniformly in $t\in[0,1]$.

A last piece of notation. Since the set of antisymmetric matrices lies in the tangent space to the free nilpotent Lie group $G$ of step $2$, at any point ${\bf z}\in G$, any smooth real-valued function $f$ defined on $G$, with compact support, has a well-defined partial differential $\partial_{\mathbb{A}}f(\mathbf z)$ in the direction of antisymmetric matrices, defined by the identity
\[
\partial_{\mathbb{A}}f(\mathbf z)(\mathbb{A}) = {\frac{\dif}{\dif t}}_{|t=0}f\big(\mathbf z+t(0,0,\mathbb A)\big),
\]
for any ${\bf z}=(1,Z,\mathbb Z)\in G$ and any antisymmetric matrix $\mathbb{A}$. Setting $\overline{\mathbf z}:=\big(1,Z,\frac{1}{2}(\bbZ+{}^t\bbZ)\big)$, we further have
\[ 
\big| f(\mathbf z) - f(\overline{\mathbf z}) - (\partial_{\mathbb{A}}f)(\overline{\mathbf z})(\mathbb{A}^{\!\mathbf z}) \big| \lesssim_f  |\mathbb{A}^{\!\mathbf z}|^2,
\]
since $f$ has compact support. Denote by $e$ the unit of the group $G$. Denote by $\mathbb{A}^{\!\bfW}$ the antisymmetric part of $\mathbb{W}$ and set $\overline{\bfX}_t := \big(1, X_t, \frac{1}{2}X_t^{\otimes 2}\big)\in G$, so that $\bfX_t = \overline\bfX_t + (0,0,\bbA^{\!\bfX})$ and $\bfW_t = \overline\bfX_t + (0,0,\bbA^{\!\bfW})$. We have, for some fixed $f$ smooth with compact support,
\begin{align*}
\left|\frac1t\widehat{\mathbb E}\big[f(\bfX_t)\right. & \left.-f(e)\big] - \frac1t\widehat{\mathbb E}\big[{f(\bfW_t)-f(e)}\big]\right| \\
&=\ \frac1t \Big| \widehat{\mathbb E}\Big[ f\big(\overline{\mathbf X}_t+(0,0,\mathbb{A}^{\!\bfX}_t)\big) - f\big(\overline{\mathbf X}_t+(0,0,\mathbb{A}^{\!\bfW}_t)\big)\Big]\Big|   \\
&\lesssim_f \frac1t \Big| \widehat{\mathbb E}\Big[ \Big((\partial_{\mathbb{A}}f)(\overline{\bfX}_t) - (\partial_{\mathbb{A}}f)(e)\Big) (\mathbb{A}^{\!\bfX}_t - \mathbb{A}^{\!\bfW}_t)\Big]\Big| + \frac1t \Big| \widehat{\mathbb E}\big[(\partial_{\mathbb{A}}f)(e) (\mathbb{A}^{\!\bfX}_t - \mathbb{A}_t^{\!\bfW})\big] \Big|   \\
&\quad\quad +\frac{1}{t} \Big( \widehat{\mathbb E}\big[|\mathbb{A}^{\!\bfX}_t|^2\big] + \widehat{\mathbb E}\big[|\mathbb{A}_t^{\!\bfW}|^2\big]\Big)   \\
&\lesssim_f (1) + (2) + (3)\text.
\end{align*}
We show that each term vanishes as $t$ goes to $0$, which implies that the two Markov processes $\bfX$ and $\bfW$ have the same generator, hence the same distribution. We have first from estimates \eqref{eq.rappelmoments} the upper bound
\begin{align*}
(1) & \leq \frac1{2t} \widehat{\mathbb E}\Big[ \sqrt t\left\| (\partial_{\mathbb{A}}f)(\overline{\bfX}_t) - (\partial_{\mathbb{A}}f)(e) \right\|^2\Big]  +  \frac1{2t} \widehat{\mathbb E}\Big[\frac1{\sqrt t} \big|\mathbb{A}^{\!\bfX}_t - \mathbb{A}_t^{\!\bfW}\big|^2\Big]   \\
&\lesssim_f \frac{1}{\sqrt t} \widehat{\mathbb E}\big[|\overline{\bfX}_t - e|^2\big] +  \frac1{t\sqrt t} \widehat{\mathbb E}\Big[\big|\mathbb{A}^{\!\bfX}_t\big|^2\Big]  +  \frac1{t\sqrt t} \widehat{\mathbb E}\Big[\big|\mathbb{A}_t^{\!\bfW}\big|^2\Big]   \\
&\lesssim_f \sqrt t\text.
\end{align*}
We also have $(2)=0$, since $\mathbb{A}^{\!\bfX}_t$ and $\mathbb{A}_t^{\!\bfW}$ are centred and $\partial_{\bbA}f(e)$ is linear. Finally, we have $(3)\lesssim t$ from the upper bounds \eqref{eq.rappelmoments}. We thus have the upper bound
\[   
\left|\frac1t\widehat{\mathbb E}\big[{f(\bfX_t)-f(e)}\big] - \frac1t\widehat{\mathbb E}\big[{f(\bfW_t)-f(e)}\big]\right| \lesssim_f \sqrt t,
\]
from which the result follows.
\end{proof}

\bigskip

\subsection{From Euclidean space to Riemannian manifolds}
\label{SubsectionManifold}

Let $(\mathcal M,g)$ be a Riemannian manifold of dimension $d$, without boundary. We emphasised in the introduction that anisotropic Brownian motion describes the random motion of a non-point-like object, with its own notion of local orientation. Such an object is represented by a point in the orthonormal frame bundle $O\mathcal M$ of $\mathcal M$, where its dynamics is described by a stochastic differential equation. We refer to Hsu's book \cite{Hsu} for a reference textbook on stochastic differential geometry.

In this subsection, we use Einstein summation convention: indices appearing twice are implicitely summed.

\medskip

\subsubsection{The orthonormal frame bundle $O\mathcal M$ of $\mathcal M$}
\label{SubsectionOrthonormalBundle}

\smallskip

Denote by $z = (q,e)$ a generic point of the orthonormal frame bundle $O\mathcal M$ of $\mathcal M$, with $q\in\mathcal M$ and $e : \bbR^d\rightarrow T_q\mathcal M$, an orthonormal frame of $T_q\mathcal M$; we write $\pi : O\mathcal M\rightarrow \mathcal M$ for the canonical projection map. The Levi-Civita connection on $T\mathcal M$ induces a notion of horizontal vectors on $T\mathcal M$ or $O\mathcal M$. Let $\mathrm H$ stand for the horizontal lift operator, meaning the map $O\mathcal M\times\bbR^d\to TO\mathcal M$ uniquely characterised by the property that $\mathrm H_z(u)\in T_zO\mathcal M$ is horizontal and
\[
\dif\pi_z\big(\mathrm H_z(u)\big)=e(u),
\] 
for any $u\in\bbR^d$ and $z=(q,e)\in O\mathcal M$. Letting $\big(\epsilon_1,\dots,\epsilon_d\big)$ be the canonical basis of $\bbR^d$, local coordinates $q^i$ on $\mathcal M$ induce canonical coordinates on $O\mathcal M$ by writing
\[
e_i := e(\epsilon_i) = e_i^j\frac{\partial}{\partial q^j}.
\]
Denoting by $\Gamma^k_{ij}$ the Christoffel symbols of the Levi-Civita connection associated with the above coordinates, the vector fields $\mathrm H(u)$ have the following expression.
\begin{equation*}
\mathrm H_z(\epsilon_\alpha) = e_\alpha^i \frac{\partial}{\partial q^i} - \Gamma_{ij}^k(q)  e_\alpha^i e_l^j \frac{\partial}{\partial e_l^k}.
\end{equation*}

\medskip

\subsubsection{Cartan's development map and anisotropic kinetic Brownian motion}

\smallskip

Roughly speaking, Cartan's development map associates in its simplest form a $\mathcal C^1$ path in $\mathcal M$, started from $q_0\in\mathcal M$, to any $\mathcal C^1$ path in the Euclidean space $\bbR^d$. Technically, given a $\mathcal C^1$ path $(x_t)_{t\geq 0}$ in $\bbR^d$, and $z_0=(q_0,e_0)\in O\mathcal M$, the Cartan development of $(x_t)_{t\geq 0}$ on $\mathcal M$ is defined as the projection $(q_t)_{0\leq t<T}$ on $\mathcal M$ of the horizontal $O\mathcal M$-valued path $(z_t) =: (q_t,e_t)_{0\leq t<T}$ solution of the ordinary differential equation
\begin{equation}\label{eq.development}
\dif z_t = \mathrm{H}_{z_t}(\dif x_t),\quad\text { i.e. }\quad\dot z_t=\mathrm H_z(\dot x_t)
\end{equation}
started from $q_0$, possibly up to some explosion time $T$. Note that the choice of $x:t\mapsto tu$ for some $u\in\bbR^d$ leads to $q$ being a geodesic with initial condition $\dot q_0=e_0(u)$; in particular, the development of $X^\sigma$ tends to a geodesic with random initial condition as $\sigma\to0$.

Classical stochastic analysis (in the Stratonovich sense) can be used to make sense of the preceding equation for $x$ a semimartingale, defining Cartan's stochastic development --- refer to Hsu's book \cite{Hsu} for a pedagogical account of the theory. For example, one of the many equivalent constructions of Brownian motion on $\mathcal M$ started at $q_0$ consists in developing a standard Euclidean Brownian motion. Accordingly, we define anisotropic Brownian motion on $\mathcal M$ as the development of the Euclidean Brownian motion with covariance $\diag(\gamma_1,\cdots,\gamma_d)$.

Anisotropic kinetic Brownian motion $(q^\sigma_t)_{0\leq t<T}$ on $\mathcal M$ is the stochastic development of the anisotropic kinetic Brownian motion $(X^\sigma_t)_{t\geq 0}$ on $\bbR^d$; it is indexed by the speed parameter $\sigma$ of its flat counterpart. This is a $\mathcal C^1$ random path which depends on the entire frame $e_0$ --- its isotropic counterpart only depends \textit{in law} on $e_0$, from symmetry properties of Wiener measure on $\bbR^d$.  Although $X^\sigma$ converges weakly to an anisotropic Brownian motion $B$ on $\bbR^d$, the poor regularity properties of the It\^o solution map does not allow to conclude that anisotropic Brownian motion $x^\sigma$ on $\mathcal M$ converges to projection on $\mathcal M$ of the solution of the equation
\[
\dif z_t = \mathrm{H}(z_t)\,\circ\dif B_t.
\]
This is exactly the kind of conclusion that rough paths theory provides.

\medskip

\subsubsection{Rough paths and rough differential equations with values in manifolds}
\label{SubsectionRoughPaths}

\smallskip

We discuss a few results of rough paths theory with values in manifolds. These results are all classical, and their Euclidean counterparts can be found e.g. in \cite{FrizHairer} or \cite{FV10}. Let $\mathcal N$ be a manifold, and, for a collection $A=(A_1,\cdots,A_n)$ of smooth vector fields on $\mathcal N$ and an initial condition $p\in\mathcal N$, consider the (deterministic) controlled differential equation
\[
\dif z_t = A(z_t)\dif x_t\text,\quad z_0 = p
\]
on $\mathcal N$, where $x$ is a driving curve with values in $\bbR^n$. The equation makes sense whenever $x$ is of class $\mathcal C^1$ (dividing each side by $\dif t$, one might say), and if moreover $x$ is of class $\mathcal C^2$, its solution is characterised by the fact that for any fixed $t\geq0$ and $f:\mathcal N\to\bbR$ smooth with compact support,
\[ f(z_t) = f(z_s) + (A_if)(z_t)(x^i_t-x^i_s) + O(|t-s|^2) \]
as $s\to t$. Now if $\bfX=(X,\bbX)$ is a rough path of Hölder regularity $1/3<\gamma\leq1/2$, we consider the following notion of solution: a continuous path $z:[0,T)\to \mathcal N$ is a solution of the rough differential equation
\begin{equation}\label{eq.rde}
\dif z_t = A(z_t)\bfX_{\dif t}\,\text,\quad z_0 = p
\end{equation}
if one can find some $a>1$ such that any choice of $t\geq0$ and $f:\mathcal N\to\bbR$ smooth with compact support yields
\[ f(z_t) = f(z_s) + (A_if)(z_t)(X^i_t-X^i_s)
                   + (A_iA_jf)(z_t)\bbX^{ij}_{ts}
                   + O(|t-s|^a) \]
as $s\to t$. This point of view is taken from \cite{BailleulIHP,BailleulFlows}, in the mindset of \cite{Davie}. We say that $z$ explodes as $t\to T$ if $z$ leaves any compact set.

In a probabilistic mindset, the fundamental remark is that, for $\bfX$ the Stratonovich rough path lift of some standard Brownian motion $W$, such a solution coincides almost surely with the solution of the Stratonovich equation
\[ \dif z_t = A(z_t)\circ\dif W_t\text,\quad z_0=p\text. \]
It is a striking feature of rough paths theory that not only does \eqref{eq.rde} admit a unique solution $z$ for any (deterministic) rough path $\bfX$, in the above sense and up to some explosion time $T>0$, but also the Itô-Lyons map $\bfX\mapsto z$ is continuous in the following sense. Fix $d$ a Riemannian distance on $\mathcal N$. If $T'<T$ and $\eps>0$, there exists some $\delta>0$ such that for any $\bfX'$ at rough path distance at most $\delta$ from $\bfX$, the solution $z'$ of
\[ \dif z'_t = A(z'_t)\bfX'_{\dif t}\,\text,\quad z'_0=p \]
is defined on $[0,T']$ and satisfies $d(z_t,z'_t)<\eps$ for all $0\leq t\leq T'$.

This kind of continuity in enough to ensure convergence in distribution: namely, if $(\bfX^n)_{n\geq0}$ is a family of random rough paths converging weakly to $\bfX$ with respect to the rough path topology, then in a sense, the (random) solution $z^n$ of \eqref{eq.rde} driven by $\bfX^n$ converges to the solution of that driven by $\bfX$. Let us make that point precise. Denote by $\widehat{\mathcal N}$ the one point compactification of $\mathcal N$ ($\widehat{\mathcal N}=\mathcal N$ if $\mathcal N$ is compact) and set $C_p$ the space of continuous paths $z:[0,1]\to\widehat{\mathcal N}$ starting at $p$ such that $z_{t+\cdot}\equiv\infty$ whenever $z_t = \infty$. Fix $d$ a Riemannian metric on $\mathcal N$ such that $d(p,p')\to\infty$ as $p'\to\infty$, and define on $C_p$ the smallest topology containing, for any $\gamma\in C_p$ and $R,\eps>0$, the set of paths $z\in C_p$ satisfying
\[ \max_{\substack{t\geq 0\\d(p,\gamma_t)\leq R}}d(z_t,\gamma_t)<\eps\text. \]
The topology does not depend on $d$, and a sequence $z^n$ of curves in $C_p$ converges to $z^\infty$ if and only if for all $R$, the curves $z^n_{\cdot\wedge\tau_R}$ stopped when they get at distance $R$ of $p$ converge uniformly to $z^\infty_{\cdot\wedge\tau_R}$. We can now state what one might call a theorem of continuity in distribution, in the following form.

\begin{theorem}\label{thm.itodistribution}
For some fixed $1/3<\gamma\leq1/2$, let $(\bfX^n)_{n\geq0}$ be a sequence of random $\gamma$-rough paths with values in $\bbR^d$, whose distributions converge weakly to that $\bfX^\infty$. These processes might be defined on different probability spaces.

Then, for any $0\leq n\leq\infty$, there exists a unique random variable $z^n$ with values in $C_p$ such that it solves the rough differential equation
\[ \dif z^n_t = A(z^n_t)\bfX^n_{\dif t}\,\text,\quad z^n_0 = p \]
almost surely up to explosion, and the distributions of $z^n$ converge to that of $z^\infty$ with respect to the topology of $C_p$ described above.
\end{theorem}

\medskip

\subsubsection{The interpolation result}

\smallskip

The proof of Theorem \ref{thm.manifold} then follows from the rough path convergence of the rough path lift $\bfX^\sigma$ of anisotropic kinetic Brownian motion $X^\sigma$ in $\bbR^d$, Theorem \ref{thm.convergenceRP}, and the continuity properties of the It\^o-Lyons solution map to rough differential equations. As in \cite{ABT}, one needs to use the stochastic and geodesic completeness of $(\mathcal M,g)$ to conclude that the convergence of the $O\mathcal M$-valued development of anisotropic kinetic Brownian in $\bbR^d$ in not only local, but that weak convergence holds true; see Proposition 2.4.3 and Lemma 2.4.4 in \cite{ABT}. Stochastic completeness refers here to the isotropic Brownian motion on $\mathcal M$. We implicitly use here the fact that for a complete and stochastically complete Riemannian manifold, the anisotropic Brownian motion on $\mathcal M$ is also stochastically complete.

\bigskip

\section{Going a bit farther}
\label{SectionFarther}

In this section, we take a step back, and see what remains of Theorem \ref{thm.manifold} in a higher level of generality. Suppose that $(v^\sigma_t)_{t\geq0}$ is of the form $v^\sigma_t = I(\overline v_{\sigma^2t})$, with $(\overline v_t)_{t\geq0}$ a càdlàg Markov process with values in some manifold $\mathcal W$ and $I:\mathcal W\to\bbR^d$ bounded continuous --- Theorem \ref{thm.manifold} deals with the case $I:\mathcal W=\mathbb S^{d-1}\hookrightarrow\bbR^d$ and $\overline v$ the anisotropic Brownian motion with time scale 1. Because the path $(x^\sigma_t)_{t\geq0}$ integrating the velocity is Lipschitz, its development on a Riemannian manifold is well-defined, and the objects described in Theorem \ref{thm.extension} make sense. We first restate and prove it, in the form of Theorem \ref{thm.final}, then discuss some examples in Section \ref{SubsectionExamples}.

\bigskip

\subsection{A more general theorem}
\label{SubsectionFinalTheorem}

This subsection is devoted to the proof of the following rewriting of Theorem \ref{thm.extension}.

\begin{theorem}\label{thm.final}
Let $(\mathcal M,g)$ be a Riemannian manifold of dimension $d$, and $(q^\sigma_t)_{t\geq0}$ a process on $\mathcal M$ whose velocity $\dot q^\sigma_t\in T_{q_t}\mathcal M$ has image $v^\sigma_t\in T_{q_0}\mathcal M\simeq\bbR^d$ under the inverse stochastic parallel transport along $q$. Suppose that, for some càdlàg Markov process $\overline v$ on a manifold $\mathcal W$, $(v^\sigma_t)_{t\geq0}$ is the continuous image of $(\overline v_{\sigma^2t})_{t\geq0}$, i.e. $v^\sigma_t=I(\overline v_{\sigma^2t})$ with $I:\mathcal W\to T_{q_0}\mathcal M$ bounded continuous. Suppose that $\overline v$ admits an invariant measure $\mu$ such that under $\mathbb P=\mathbb P_\mu$,
\begin{enumerate}
\item \label{con.mixing} equation \eqref{eq.mix} holds with $\mathcal F_{[a,b]}$ the $\sigma$-algebra generated by $\{I(\overline v_t)\}_{a\leq t<b}$;
\item \label{con.symm} for all $1\leq i\leq d$, the flippings $(v^1,\cdots,v^{i-1},-v^i,v^{i+1},\cdots,v^d)$ have the same distribution as $v=v^\sigma=(v^1,\cdots,v^d)$ for some, hence all, $\sigma>0$.
\end{enumerate}
Then as $\sigma\to\infty$, the time rescaled process $(q^\sigma_{\sigma^2t})_{t\in[0,1]}$ converges in law to an anisotropic Brownian motion on $\mathcal M$ with covariance $\diag(\gamma_1,\cdots,\gamma_d)$,
\[
\gamma_i := \int_0^\infty\bbE\left[I(\overline v_0)^iI(\overline v_t)^i\right]\dif t\text.
\]
\end{theorem}

\begin{remark}
\label{rem.symm}
Condition \eqref{con.symm} above is indeed necessary. Assuming only that $I(\overline v)$ is centred, the tightness result stated in Corollary \ref{cor.tightness} still holds, as well as the Brownian behaviour of the Euclidean path as shown in Proposition \ref{prop.convergencewk}; however, the limit rough path needs not be Brownian --- see example \ref{ex.spinningmotion} below. In particular, there is no reason for the manifold-valued result to hold. In the common `rolling without slipping' analogy used to described stochastic development of Brownian motion, one might think of the resulting non-Brownian effect as a force rotating the paper around the contact point, so that the path on the manifold may have a tendency to lean to one side.
\end{remark}

\begin{remark}
Throughout our study, we have worked at equilibrium, with $\mathbb P=\mathbb P_\mu$. Although it simplifies the proofs, it is merely a cosmetic concern in the case of kinetic Brownian motion. In fact, under the assumption \eqref{eq.exponentialconv}, Theorem \ref{thm.final} holds for any $\mathbb P_\lambda$: see Proposition \ref{prop.outofeq} below. For instance, it will be the case in examples \ref{ex.randomflight} and, to some extent, \ref{ex.langevin} below. It is not clear whether the result should hold without this additional property.
\end{remark}

To establish Theorem \ref{thm.final}, let us review the ingredients of the proof of Theorem \ref{thm.manifold}. The tightness results, more specifically Corollary \ref{cor.tightness}, are essentially a consequence of Lemma \ref{lem.readytouse}. It holds whenever \eqref{eq.mix} is satisfied (condition \eqref{con.mixing}), $I(\overline v_0)$ is centred (condition \eqref{con.symm}) and $I$ is bounded. On the other hand, the convergence towards Brownian motion relies, in addition, on the symmetry property (condition \eqref{con.symm}) and independence of the increments. Equation \eqref{eq.mix} ensures the latter, so that the proof of Theorem \ref{thm.final} is essentially that of Theorem \ref{thm.manifold}.

\medskip

\begin{proposition}\label{prop.outofeq}
Replace condition \eqref{con.mixing} in Theorem \ref{thm.final} by the following variant of \eqref{eq.exponentialconv}. There exists some mixing time $\tau>0$ such that for all $x\in\mathcal W$ and $t>0$,
\begin{equation}\label{eq.outofeq}
\|P_t^*\delta_x - \mu\|_\mathrm{TV} \leq f(x)\exp(-t/\tau)
\end{equation}
for some function $f:\mathcal W\to\bbR_+$ integrable with respect to $\mu$.

Then the conclusion also holds under $\mathbb P_\lambda$, for any probability measure $\lambda$ on $\mathcal W$ such that $\lambda(f):=\int f\dif\lambda<\infty$.
\end{proposition}

\begin{proof}
It is enough to show the convergence of the Euclidean rough paths $(\bfX^\sigma)_{\sigma>0}$.

\smallskip

\emph{Tightness.} We claim that Proposition \ref{prop.exponentialdec} holds for $\bbE_\lambda$. Indeed, by the same arguments, we see that
\begin{equation}\label{eq.Goutofeq}
|\bbE_{P_u^*\delta_x}[G] - \bbE_\mu[G]| \leq |G|_\infty\,f(x)\,\E^{-u/\tau}
\end{equation}
holds in lieu of \eqref{eq.Gdecorrelation}. From this we deduce
\begin{align*}
       |\bbE_{P_{t-s}^*\delta_x}[G] - \bbE_{P_t^*\lambda}[G]|
& \leq |\bbE_{P_{t-s}^*\delta_x}[G] - \bbE_\mu[G]|
     + |\bbE_\mu[G] - \bbE_{P_t^*\lambda}[G]| \\
& \leq (f(x)+\lambda(f)\,\E^{-s/\tau})|G|_\infty\,\E^{-(t-s)/\tau} \text,
\end{align*}
which is enough for rest of the proof to hold. It is then an easy exercise to adapt the proof of Corollary \ref{prop.mixingcoords}, and from this point every idea leading to tightness is the same, even if some care must be given to non-stationarity in the actual computations, e.g. regarding equation \eqref{eq.kolmogorovsimple}.

\smallskip

\emph{Brownian limit.} Let $\widehat{\mathbb P}_\mu$ be the law of the Brownian rough path on $\mathsf{RP}(\gamma)$, and $\widehat{\mathbb P}_\lambda$ a limit point of the laws of $\mathbf X^\sigma$ under $\mathbb P_\lambda$. We only need to show that $\widehat{\mathbb P}_\lambda = \widehat{\mathbb P}_\mu$.

Define the translation operator $T_h$ on $\mathsf{RP}(\gamma)$ as
\[ T_h(Y,\mathbb Y) := (Y_{h+\cdot}-Y_h,\mathbb Y_{h+\cdot,h+\cdot})\text. \]
Now, for any continuous bounded map $F:\mathcal C([0,1],G)\to\bbR$ and $\eps>0$, equation \eqref{eq.Goutofeq} gives
\[
     \left|\bbE_\lambda[F(T_\eps\mathbf X^\sigma)] - \bbE_\mu[F(T_\eps\mathbf X^\sigma)] \right|
\leq \lambda(f)|F|_\infty\,\E^{-\sigma^4\eps/\tau} \text,
\]
which, taking limits along a proper subsequence, implies that
\[ \widehat\bbE_\lambda[F(T_\eps\bfX)] = \widehat\bbE_\mu[F(T_\eps\bfX)]\text. \]
But $T_\eps\mathbf X\to\mathbf X$ in $\mathcal C^0([0,1],G)$, so the above equation holds for $\eps=0$, and $\widehat{\mathbb P}_\lambda$ is the law of the announced anisotropic Brownian motion. Note that $T_\eps\mathbf X$ has no reason to converge to $\mathbf X$ in the rough path topology, so tightness had to be proved beforehand.
\end{proof}

\bigskip

\subsection{Examples}
\label{SubsectionExamples}

In this last section, we finally discuss some examples and counterexamples to the statement of Theorem \ref{thm.final}.

\smallskip

\subsubsection{Spinning motion}\label{ex.spinningmotion}
The first example illustrates what happens when the motion does not satisfy the symmetry condition \eqref{con.symm} in Theorem \ref{thm.final} above. Set $I:\mathcal W=\bbR/2\pi\bbZ\to\mathbb C\simeq\bbR^2$ the exponential $v\mapsto \E^{\I v}$, and define $\overline v$ as the spinning motion
\[ \dif\overline v_t = \dif t + \dif W_t\text,
   \quad\text{ i.e. }\quad
   \overline v_t = \overline v_0+t+W_t\ \mathrm{(mod\ }2\pi\mathrm), \]
where $W$ is a standard Brownian motion on $\bbR$. Its dynamics is of course very simple: it admits a unique invariant measure $\mu(\dif v)=\frac1{2\pi}\dif v$ and satisfies equation \eqref{eq.exponentialconv}, so all hypotheses but condition \eqref{con.symm} in Theorem \ref{thm.final} above are satisfied.

As mentioned in Remark \ref{rem.symm} above, the laws of $(X^\sigma)_{\sigma>0}$ do converge to that of a Brownian process. As of those of the lifts $(\mathbf X^\sigma)_{\sigma>0}$, however, some drift appears in the limit. Indeed, setting $\bbA^\sigma$ the antisymmetric part of $\bbX^\sigma$,
\[ (\bbA^\sigma_{t0})^{12}
 = \frac1{2\sigma^4}
   \int_0^{\sigma^4t}\int_0^s\sin(\overline v_s-\overline v_u)\dif s\dif u
 = \int_0^\infty\int_0^\infty\frac1{2\sigma^4}\Ind_{u+\tau\leq\sigma^4t}
   \sin(\overline v_{u+\tau}-\overline v_u)\dif u\dif \tau\text, \]
so we get
\[ \bbE\big[(\bbA^\sigma_{t0})^{12}\big]
 = \int_0^\infty\int_0^\infty\frac1{2\sigma^4}\Ind_{u+\tau\leq\sigma^4t}
   \sin(\tau)\,\E^{-\tau/2}\dif u\dif \tau
 = \frac12\int_0^\infty\left(t-\frac\tau{\sigma^4}\right)_+\sin(\tau)\,\E^{-\tau/2}\dif\tau \]
with $(\cdot)_+$ the positive part. The limit is a non-zero linear function of $t$, so the limit of the lifts cannot be Brownian.

\medskip

Such drift phenomena in the Lévy area have arisen and been studied in different works recently, particularly in the context of random walks. See e.g. the articles \cite{LS1,LS2} of Lopusanschi and Simon, and those of Ishiwata, Kawabi and Namba, \cite{IKN1,IKN2}.

\medskip

\subsubsection{Random flight}\label{ex.randomflight}
The case where $I:\mathcal W=\mathbb S^{d-1}\hookrightarrow\bbR^d$ and $\overline v$ is a pure jump process, with rate 1 and uniform measure, is the so-called random flight studied by Pinsky in \cite{Pinsky}, where it is called the isotropic transport process. In this case, the mixing property \eqref{eq.mix} is a consequence of the stronger statement \eqref{eq.exponentialconv} that the dynamics converges exponentially fast to equilibrium in total variation, in the same way we treated anisotropic Brownian motion. There are no complications in dealing with jumps.

Because the velocity is isotropic, the limit covariance $\diag(\gamma_1,\cdots,\gamma_d)$ is proportional to $\mathrm{Id}$. Setting $T$ the first jump time,
\[ \gamma_i = \frac2d\int_0^\infty\bbE[\overline v_0\cdot\overline v_t]\dif t
            = \frac2d\int_0^\infty\mathbb P(T\leq t)\dif t
            = \frac2d\text, \]
and we recover the result of \cite{Pinsky}.

\medskip

\subsubsection{Donsker invariance principle for random walks}
\label{ex.donskerindependent}
Another example, studied in \cite{Breuillard} by Breuillard, Friz and Huesmann, is that of random walks. If $(Y_k)_{k\geq0}$ is a sequence of independent \emph{bounded} random variables with values in $\bbR^d$, \emph{symmetric} in the sense that their common law is invariant with respect to the flippings as described in condition \eqref{con.symm} in Theorem \ref{thm.final} above, we can consider the piecewise linear processes $W^\sigma$ defined by
\[ W^\sigma:t=\frac{n+u}{\sigma^4}\mapsto\frac1{\sqrt n}\sum_{k<n}Y_k+uY_n\text,\quad n\in\bbN\text,\ u\in[0,1)\text. \]
Let us translate this dynamics in our framework. Set $\mathcal W = \bbR/\bbZ\times\bbR^d$, $I:(\alpha,y)\mapsto y$, and define the dynamics of $(\overline v_t)_{t\geq0}=(\alpha_t,y_t)_{t\geq0}$ as follows. Given initial conditions $(\alpha_0,y_0)\in[0,1)\times\bbR^d$, $\alpha$ grows continuously with rate 1, i.e. $\alpha_t = \alpha_0 + t\ \mathrm{(mod\ 1)}$, whereas $y$ stays constant on time intervals of length 1, then jumps independently of the past according to the law of $Y_k$, i.e. $y_t = Y_{\lfloor t-\alpha_0\rfloor} $ with the convention $Y_{-1}=y_0$. With initial condition $\delta_0\otimes\mathcal L(Y_0)$, we see that the law of $x^\sigma$ is exactly that of $W^\sigma$.

The process $\overline v$ is Markovian, although not Feller, and admits an invariant measure $\mathrm{Unif}(\bbR/\bbZ)\otimes\mathcal L(Y_0)$. Because it is not ergodic, there is no hope for equation \eqref{eq.exponentialconv} to hold. Maybe surprinsingly, even if $I$ kills the non-mixing coordinate, it is also false that condition \eqref{con.mixing} of Theorem \ref{thm.final} holds: in the case where $Y$ has no atoms, take $P$ to be the first jump time in $[0,1]$, and $F$ the first jump time in $[n,n+1]$. However, it is true for any $\mu_\alpha:=\delta_\alpha\otimes\mathcal L(Y_0)$, with constants independent of $\alpha$ --- indeed, it is obvious that for $t>1$ and any probability law $\lambda$ on $\bbR^d$,
\[
P_t^*(\delta_\alpha\otimes\lambda) = \mu_{\alpha+t} = P_t^*\mu_{\alpha}
\]
holds in lieu of \eqref{eq.exponentialconv}. Remarkably, nothing more than this is needed throughout the proof. It should be clear that Proposition \ref{prop.mixingcoords} holds for any $\mu_\alpha$, and that tightness follows in the same fashion. Independence of increments, as stated in Propositions \ref{prop.convergencewk} and Theorem \ref{thm.convergenceRP}, hides no difficulty either. It is true that one has to be careful about the limit variance in Proposition \ref{prop.convergencewk}, because the Markov property is used in a crucial way. In our case, for any $\alpha\in[0,1)$ and $\sigma>1$, we end up with
\begin{align*}
      \mathbb{E}_{\mu_\alpha}\big[\big((X^\sigma_1)^i\big)^2\big]
   &= \frac1{\sigma^4}\int_0^{\sigma^4}\int_0^{\sigma^4}\mathbb{E}_{\mu_\alpha}[y^i_sy^i_t]\,\dif s\,\dif t \\
   &= \sum_{n\geq-1}\frac1{\sigma^4}\int_0^{\sigma^4}\int_0^{\sigma^4}\Ind_{n+\alpha\leq s,t<n+1+\alpha}\bbE_{\mu_{\alpha}}[y^i_sy^i_t]\,\dif s\,\dif t   \\
   &= \bbE[|Y^i_0|^2]\cdot\frac{(1-\alpha)^2 + \lfloor\sigma^4+\alpha-1\rfloor + \{\sigma^4+\alpha-1\}^2}{\sigma^4}
\end{align*}
with $\{\cdot\}$ the fractional part. In the limit, the variance converges to $\bbE[|Y_0^i|^2]$, and the result of Theorem \ref{thm.manifold} holds with covariance $\bbE[Y_0Y_0^*]$, in accordance with \cite{Breuillard}.

\medskip

Surprisingly enough, the symmetry condition \eqref{con.symm} is not mandatory here: see \cite{Breuillard}. In particular, drift in the antisymmetric part of $\bbX^\sigma$ that does not vanish in the limit, as mentioned in example \ref{ex.spinningmotion} above, comes from additional structure: in \cite{LS1}, the hidden Markov chain; in \cite{LS2}, the underlying directed graph; etc.

Note that in the case of random walks, as a consequence of the work of Chevyrev, see \cite[Example 5.8]{Chevyrev}, convergence of $X^\sigma$ as stated in Proposition \ref{prop.convergencewk} in enough to ensure convergence of $\mathbf X^\sigma$ to some random rough path. It is not clear from this approach, however, that this limit is indeed Brownian.

\medskip

\subsubsection{Donsker invariance principle for Markov chains}
\label{ex.donskermarkov}
The reader may have noticed that in the above example \ref{ex.donskerindependent}, independence of the variables $(Y_k)_{k\geq0}$ is a bit much, and one could work with covariances vanishing exponentially fast. Suppose for instance that $(Y_k)_{k\geq0}$ is a time-homogeneous Markov chain with invariant measure $\mu$ with compact support, whose correlations decrease as $\E^{-k/\tau}$, $\tau>0$; namely, letting $Q$ be the transition kernel of $Y$,
\[ \|\delta_y\,Q^k - \mu\|_\mathrm{TV} \lesssim \E^{-k/\tau} \]
for all $y$ in the support of $\mu$. Then, setting $\mu_0:=\delta_0\otimes\mu$, we get, for any probability measure $\lambda$ with $\mathrm{Supp}\ \lambda\subset\mathrm{Supp}\ \mu$,
\[ \|P_t^*(\delta_0\otimes\lambda) - P_t^*\mu_\alpha\|_\mathrm{TV}
 = \left\|\int(\delta_y\,Q^{\lfloor t\rfloor} - \mu)\lambda(\dif y)\right\|_\mathrm{TV}
\lesssim \E^{-t/\tau}\text. \]
Again, this inequality can be substituted for equation \eqref{eq.exponentialconv} in the proof of Proposition \ref{prop.exponentialdec}, and under the same symmetry condition as above, the convergence result still holds true.

Examples of such Markov chains are any aperiodic irreducible finite state Markov chain; or any Markov chain with transition kernel $Q(y,\dif y')$ absolutely continuous with respect to some measure $\nu$, and such that $\frac{\dif Q(y,\cdot)}{\dif\nu}$ is bounded below by a positive constant $m>0$, uniformly in $y,y'$. Note however that the symmetry condition \eqref{con.symm} of Theorem \ref{thm.final} is a bit stronger than in the independent case, since we need the flippings to leave the law of the whole sequence invariant.

\medskip

\subsubsection{Time-dependent Brownian motion}\label{ex.timedependent}
The way we wrote our convergence theorems is ill-suited to treat time-dependent randomness. However, there are cases where randomness can be somewhat dissociated from the time dependence, and our methods do in fact yield interesting convergence results. In the present example, we set to recover, in the limit, the Brownian motion on a manifold $\mathcal M$ endowed with a time-dependent metric $g_t$, as introduced in \cite{ArnaudonThalmaier} by Arnaudon, Coulibaly and Thalmaier.

Such an approach has already been set up in \cite{Kuwada}, in a similar fashion as the random flight described in example \ref{ex.randomflight} above. The idea is to freeze the metric in small time intervals $[t_i,t_{t+1}]$, say of size $1/\sigma^4$, over which the movement $q$ is purely geodesic with respect to the metric $g_{t_i}$, the initial condition being chosen uniformly at $t_i$ on the unit $g_{t_i}$-sphere of the tangent space of $\mathcal M$ at $q_{t_i}$. Suitably renormalised, this process converges to the time-dependent Brownian motion described above. We introduce a similar random flight which lets the metric vary continuously, and may be considered more natural in this respect, then prove its convergence to time-dependent Brownian motion.

\medskip

We begin by describing time-dependent Brownian motion and its surroundings. Suppose $g_t$ is smooth, as a function on $\bbR_+\times T\mathcal M\otimes T\mathcal M$. Let $F\mathcal M$ be the frame bundle over $\mathcal M$, and choose a point $q_0\in\mathcal M$ together with a $g_0$-orthonormal frame $e_0$ of $T_{g_0}\mathcal M$. For a $\mathcal C^1$ path $(x_t)_{t\geq0}$ in $\bbR^d$, we define the time-dependent development of $x$ as the solution $(z_t)_{t\geq0}=(q_t,e_t)_{t\geq0}$ of the following equation, whose terms we describe below.
\begin{equation}\label{eq.Tdevelopment}
\dif z_t = \mathrm H_{t,z_t}(\dif X_t)
         - \frac12\frac{\partial g_t}{\partial t}(u_t\epsilon_i,u_t\epsilon_j)\mathrm V^{ij}_{z_t}\dif t,
\quad z_0 = (q_0,e_0).\end{equation}
We use Einstein notation. As in Section \ref{SubsectionOrthonormalBundle}, $(\epsilon_1,\cdots,\epsilon)$ is the canonical basis of $\bbR^d$, and the $\mathrm H_{t,z}\epsilon_i$, resp. $\mathrm V^{ij}_z$, are the canonical horizontal vector fields, resp. vertical vector fields. Note that because the metric $g$ is time-dependent, the associated horizontal vector fields $\mathrm H$ must depend on $t$ as well. In coordinates,
\[ \mathrm H_{t,z}(\epsilon_\alpha) = e_\alpha^i \frac{\partial}{\partial q^i} - \big(\Gamma_t(q)\big)_{ij}^k\, e_\alpha^i e_l^j \frac{\partial}{\partial e_l^k},\qquad
   \mathrm V^{ij}_z = e_j^k \frac\partial{\partial e_i^k}\text. \]
If we compare \eqref{eq.Tdevelopment} to \eqref{eq.development}, the added vertical fields are there to ensure that $e_t$ is at all times orthonormal for $g_t$. We refer to \cite{Coulibaly} for an insight about why this definition is a sensible choice.

In particular, the time-dependent geodesics are the solutions of the equation associated to $ x_t=tu $ for some fixed $u\in\bbR^d$, and the time-dependent Brownian motion is the solution driven by some standard Brownian motion $W$ in the Stratonovich sense, or, equivalently, by the standard Stratonovich rough path $\mathbf W$ in the rough sense.

Note that we did not discuss time-dependent rough differential equations in Section \ref{SubsectionRoughPaths}. In the case of an equation driven by a $\mathcal C^1$ control $x$, the standard technique is of course to consider $t\mapsto(t,x_t)$ as the control. The same trick works with rough paths: associated to any rough path $\mathbf Y=(Y,\mathbb Y)$ is a canonical lift $\widehat{\mathbf Y}$ of $t\mapsto(t,Y_t)$ compatible with $\mathbf Y$. The solution of time-dependent rough differential equations is then well-defined. In what follows, we will also use the fact that $\mathbf Y\mapsto\widehat{\mathbf Y}$ is continuous in the rough path topology, so $\widehat\bfX^\sigma\to\widehat\bfX$ weakly whenever $\bfX^\sigma\to\bfX$ weakly.

\medskip

We define a kind of interpolated random walk on $\mathcal M$ whose limit will be the Brownian motion described above. Fix $\sigma>0$, and define $W^\sigma$ successively on the intervals $[s,t]=[\frac n{\sigma^4},\frac{n+1}{\sigma^4}]$ as follows: $\xi^\sigma_n$ is chosen independently of all the rest according to the uniform measure on the unit $g_s$-sphere of $T_{W^\sigma_s}\mathcal M$, and $W^\sigma$ on $[s,t]$ is a time-dependent geodesic in the above sense, with initial condition $\dot W^\sigma_s = \sqrt d\,\xi^\sigma_n$.

As in the previous example, there is a direct equivalent of this dynamics in our framework. Set $\mathcal W=\bbR/\bbZ\times \mathbb S^{d-1}$ and $I:(\alpha,y)\hookrightarrow \sqrt dy$, following the same dynamics as in \ref{ex.donskerindependent}, with $Y_0$ uniformly distributed on $\mathbb S^{d-1}$. We choose the initial condition to be $\delta_0\otimes\mathrm{Unif}(\mathbb S^{d-1})$; for the same reasons as in example \ref{ex.donskerindependent} above, $(\mathbf X^\sigma)_{\sigma>0}$ converges to the Brownian rough path $\mathbf X$ with covariance $d\,\bbE[Y_0Y_0^*] = \mathrm{Id}$.

Everything described so far is essentially time-invariant --- the time-dependence appears when we use this family of rough paths to describe a motion on $\mathcal M$. Fix $q_0\in\mathcal M$, and $e_0$ a $g_0$-orthonormal frame of $T_{q_0}\mathcal M$. Define the solution $(z_t)_{t\geq0}=(q_t,e_t)_{t\geq0}$ (up to explosion) on the frame bundle $F\mathcal M$ of equation \eqref{eq.Tdevelopment} driven by $\mathbf X$, in the rough sense.

By definition, $q_t$ defined as above is the Brownian motion associated to the time-dependent metric $g_t$, as described in \cite{ArnaudonThalmaier}. If we set $z^\sigma=(q^\sigma,u^\sigma)$ the solution of the equation driven by $\mathbf X^\sigma$, we get instead $q^\sigma = W^\sigma$ in law. The convergence of $\mathbf X^\sigma$, together with the general theory of rough paths (see Theorem \ref{thm.itodistribution}), ensures that $q^\sigma$, hence $W^\sigma$, converges weakly to the time-dependent Brownian motion $q$.

\medskip

\subsubsection{Langevin Process}\label{ex.langevin}
We conclude with an example where the velocity $v$ has unbounded support. We consider the process with anisotropic Ornstein-Uhlenbeck velocity, i.e. satisfying
\[ \dif\overline v_t = - \overline v_t\dif t + \dif B_t \]
for $B$ an anisotropic Brownian motion of covariance $\Sigma$. In the isotropic case, it is a simple scalar example of the hypoelliptic Laplacian of Bismut; see \cite{Bismut}. The anisotropic case is also treated in \cite{BHVW}.

Here, $I:\mathcal W=\bbR^d\to\bbR^d$ is simply the identity, and hence does quite fit the hypotheses of Theorem \ref{thm.extension}. However, it is well known that $\overline v$ admits as an invariant measure $\mu=\mathcal N(0,\frac12\Sigma)$ the Gaussian distribution with covariance $\frac12\Sigma$. Using the coupling $B'_t=-B_t$, it is known, and not difficult to see, that
\[ \|P_t^*\delta_x-\mu\|_\mathrm{TV}\lesssim (1\vee|x|)\,\E^{-t}\text, \]
from whence, because $1\vee|x|$ is in $L^1(\mu)$, we derive Proposition \ref{prop.exponentialdec}; see Proposition \ref{prop.outofeq}.

In our proof, boundedness of the velocity is essentially used twice: for proving the decorrelation of coordinates in Proposition \ref{prop.mixingcoords}, and to show that the variance of the limit must be the limit of the variances in \ref{prop.convergencewk}. Because $\mu$ has moments of all order, the latter will add no difficulty --- in fact, any moment of order $>2$ would suffice. As for the former, it is a bit trickier. We use the following variation of Proposition \ref{prop.mixingcoords}.

\begin{proposition} Fix some $\eps>0$ and some positive integer $n\in\bbN^*$. There exists $\tau'=\tau'(\tau,n,\eps)>0$ such that under $\mathbb P=\mathbb P_\mu$, and for any indices $1\leq j_1,\cdots,j_n\leq d$ and times $s_1,\cdots,s_n\geq0$,
\[
\left| \bbE\big[v_{s_1}^{j_1}\cdots v_{s_1+\cdots+s_n}^{j_n}\big]\right| \lesssim |v^{j_1}_0|_{L^{n+\eps}}\cdots|v^{j_n}_0|_{L^{n+\eps}}\,\E^{-\Delta/{\tau'}}.
\]
\end{proposition}

We give only hints of the proof. In the spirit of the proof of Proposition \ref{prop.mixingcoords}, set
\[ V_- := \prod_{1\leq k<k_0}\left(v^{j_k}_{t_k}/|v^{j_k}_0|_{L^{n+\eps}}\right) \]
and similarly for $V_0$ and $V_+$. Write $V_*=W_*+R_*$ with $W_*:=V_*\Ind_{|V_*|\geq M}$; for $M=\exp(\eta\Delta)$ with $\eta>0$ small enough, the proof of Proposition \ref{prop.mixingcoords} applied to $W_*$, together with a careful handling of the remainder $R_*$, are enough to get to the above result. It automatically implies Lemma \ref{lem.readytouse}, since $\mu$ has moments of all order, hence the conclusion of Theorem \ref{thm.manifold}.

Note that the treatment of unboundedness is not specifically designed for the Langevin process, so it can be applied to the study of the random walk as well. Moreover, it is not necessary for all moments to exist: moments of order $\alpha>2/(1-2\gamma)$ are enough to ensure tightness in $\mathsf{RP}(\gamma)$. Indeed, our proof, enhanced by the above corollary, will hold with moments of order $2n>2/(1-2\gamma)$ for any positive integer $n$; but adding an easy truncation argument at the beginning of the proofs of Lemma \ref{lem.kolmogorovwk} and \ref{lem.kolmogorovRP} will strengthen the result to non even integral moments. In this respect, our moment assumption is a bit weaker than that of \cite{Breuillard} in the symmetrical case.

\bigskip
\bigskip

\newcommand{\etalchar}[1]{$^{#1}$}

\end{document}